\documentclass[11pt]{amsart}
\usepackage[utf8]{inputenc}
\usepackage{color}
\usepackage{url}
\usepackage{amsmath,amsthm,amsfonts,amssymb}
\usepackage[color,matrix,arrow]{xy}
\usepackage{graphicx}
\usepackage{stmaryrd}
\usepackage{algorithm}
\usepackage{algpseudocode}

\usepackage{tabularx, makecell}
\usepackage[shortlabels]{enumitem}
\usepackage{multirow} 
\usepackage{caption}
\usepackage{float}
\usepackage{tikz-cd}

\theoremstyle{plain}
\newtheorem{theorem}{Theorem}[section]
\newtheorem*{theorem*}{Theorem}
\newtheorem{lemma}[theorem]{Lemma}
\newtheorem{corollary}[theorem]{Corollary}
\newtheorem{proposition}[theorem]{Proposition}

\theoremstyle{definition}
\newtheorem{example}[theorem]{Example}
\newtheorem{definition}[theorem]{Definition}
\theoremstyle{remark}
\newtheorem{remark}[theorem]{Remark}

\def\core{{\mathrm{core}}}

\def\lk{{\mathrm{lk}}}
\def\st{{\mathrm{st}}}

\def\Loc{{\mathrm{Loc}}}

\def\Crit{{\mathrm{Crit}}}
\def\QuotientClass{{\mathrm{QuotientClass}}}

\def\X{{\mathcal X}}

\def\K{{\mathcal K}}

\def\J{{\mathcal J}}
\def\K{{\mathcal K}}

\def\Q{{\mathcal Q}}

\def\RR{{\mathbb R}}

\def\ZZ{{\mathbb Z}}

\newcommand\Nplusonedef{\hspace{2 pt}\diagup\hspace{-4.8 pt} \searrow\hspace{-9pt}^{^{N+1}} \hspace{5 pt}} 
\newcommand\Ndef{\hspace{2 pt}\diagup\hspace{-4.8 pt} \searrow\hspace{-9pt}^{^{N}} \hspace{5 pt}}

\newcommand\threedef{\hspace{2 pt}\diagup\hspace{-5 pt} \searrow\hspace{-8pt}^3 \hspace{5 pt}}

\newcommand\ce{{\hspace{2.4 pt} \searrow\hspace{-8 pt}^e \hspace{5 pt}}}

\newcommand\co{{\hspace{2 pt}\searrow \hspace{3 pt}}}
\newcommand\ex{{\nearrow \hspace{3 pt}}}
\newcommand\esc{{\searrow\hspace{-6 pt}\searrow\hspace{-8 pt}^e \hspace{5 pt}}}
\newcommand\sco{{\searrow \hspace{-6 pt}\searrow\hspace{3 pt}}}

\begin{document}

\title{Strong discrete Morse theory}

\author{Ximena Fernández}
\address{Department of Mathematics,  City St George's University of London, and Mathematical Institute, University of Oxford,  United Kingdom}
\email{ximena.fernandez@city.ac.uk}


% MSC 2020 Subject Classification:
% 55U10 – Simplicial sets and complexes
% 57Q05 – Simple homotopy theory and collapses
% 05E45 – Combinatorial aspects of simplicial complexes
% 68W05 – Algorithms; nonnumeric algorithms


\begin{abstract}
The purpose of this work is to develop a version of Forman's discrete Morse theory for simplicial complexes, based on \textit{internal strong collapses}. Classical discrete Morse theory can be viewed as a generalization of Whitehead’s collapses, where each Morse function on a simplicial complex $K$ defines a sequence of elementary internal collapses. This reduction guarantees the existence of a CW-complex that is homotopy equivalent to $K$, with cells corresponding to the critical simplices of the Morse function. However, this approach lacks an explicit combinatorial description of the attaching maps, which limits the reconstruction of the homotopy type of $K$.
By restricting discrete Morse functions to those induced by total orders on the vertices, we develop a \textit{strong discrete Morse theory}, generalizing the strong collapses introduced by Barmak and Minian. We show that, in this setting, the resulting reduced CW-complex is regular, enabling us to recover its homotopy type combinatorially. We also provide an algorithm to compute this reduction and apply it to obtain efficient structures for complexes in the library of triangulations by Benedetti and Lutz.
\end{abstract}

\maketitle

\section{Introduction}

The combinatorial analysis of homotopy types dates back to the 1940s with Whitehead’s simple homotopy theory~\cite{whitehead_39, whitehead_41, whitehead_49_1, whitehead_50}. In an effort to describe homotopy equivalences through rigid transformations of combinatorial structures, Whitehead introduced collapses and expansions for finite simplicial complexes. These moves were inspired by Tietze transformations for group presentations. While Tietze’s theory provides a complete set of transformations between equivalent presentations, Whitehead’s simple homotopy theory is more restrictive and fails to capture all homotopy deformations. Nevertheless, it remains a powerful combinatorial framework, influencing classical problems~\cite{akbulut_kirby_85, milnor_66}; such as Zeeman’s conjecture~\cite{zeeman_63}; the Andrews--Curtis conjecture~\cite{andrews_curtis_65, hog-angeloni_metzler_93}; and the Poincaré conjecture~\cite{poincare_1904, zeeman_poincare_63}; as well as modern computational approaches to analysing complexes and data~\cite{bauer17, benedetti_24, Robins_2011}.

Whitehead's elementary deformations are called \textit{collapses} (and their reverse moves, \textit{expansions}), which rely on the concept of a \textit{free face}. If $\tau$ is the unique simplex of a simplicial complex $K$ properly containing $\sigma$, then both $\sigma$ and $\tau$ can be removed from $K$ while preserving its homotopy type and simplicial structure (see Figure~\ref{fig:collapse}, left). However, many simplicial complexes lack free faces (e.g., triangulations of closed manifolds). A more general concept is that of an \textit{internal free face}, leading to an \textit{internal collapse}~\cite{kaczynski1998, kozlov_20, fernandez_24}. Here, $\sigma$ becomes a free face of $\tau$ in a \textit{subcomplex} $L \subseteq K$. To preserve the homotopy type of $K$ after removing $\sigma$ and $\tau$, the simplices in $K \smallsetminus L$ must be re-attached, often losing the simplicial structure (see Figure~\ref{fig:collapse}, right). In general, there is no explicit or computable method to determine such attachments in arbitrary dimensions.

\begin{figure}[htb!]
\hspace{-5pt}\includegraphics[width=0.45\textwidth]{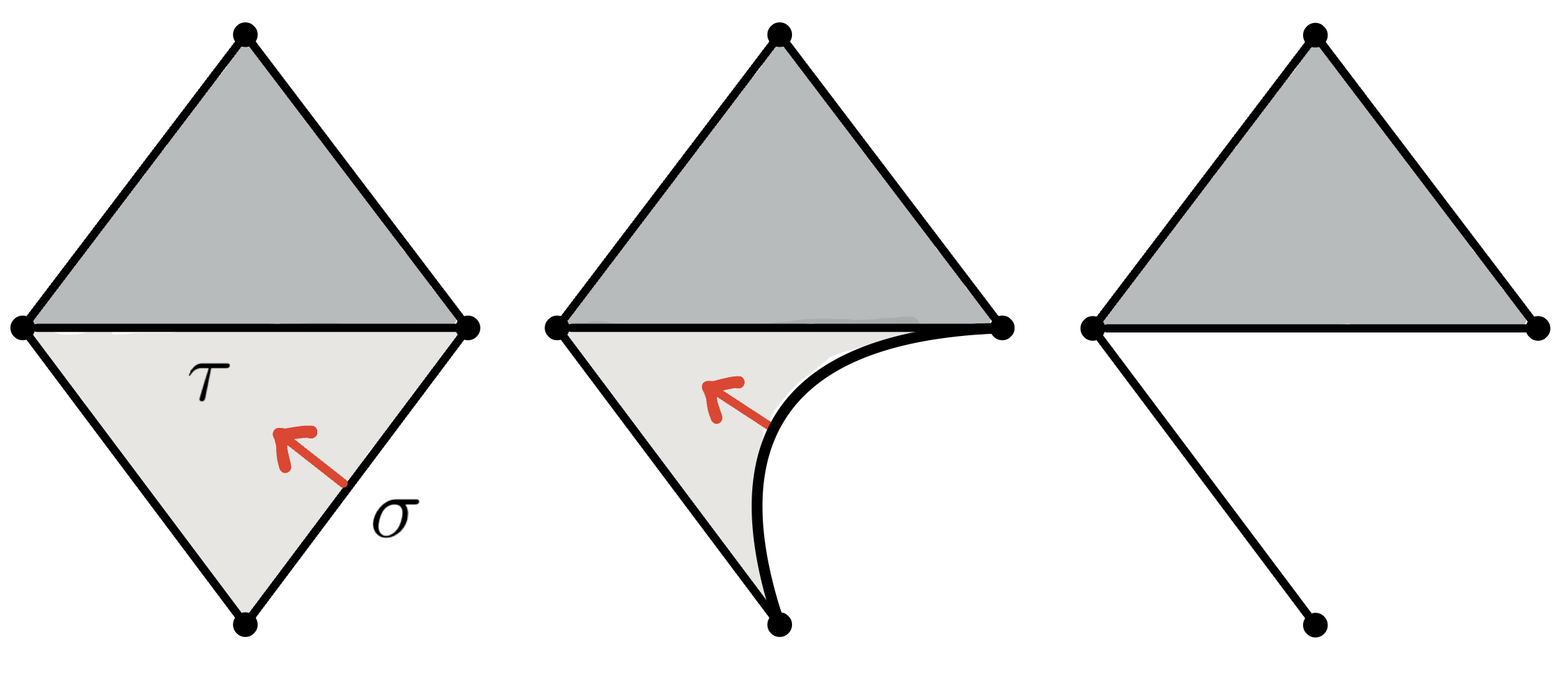}\hspace{40pt}\includegraphics[width=0.45\textwidth]{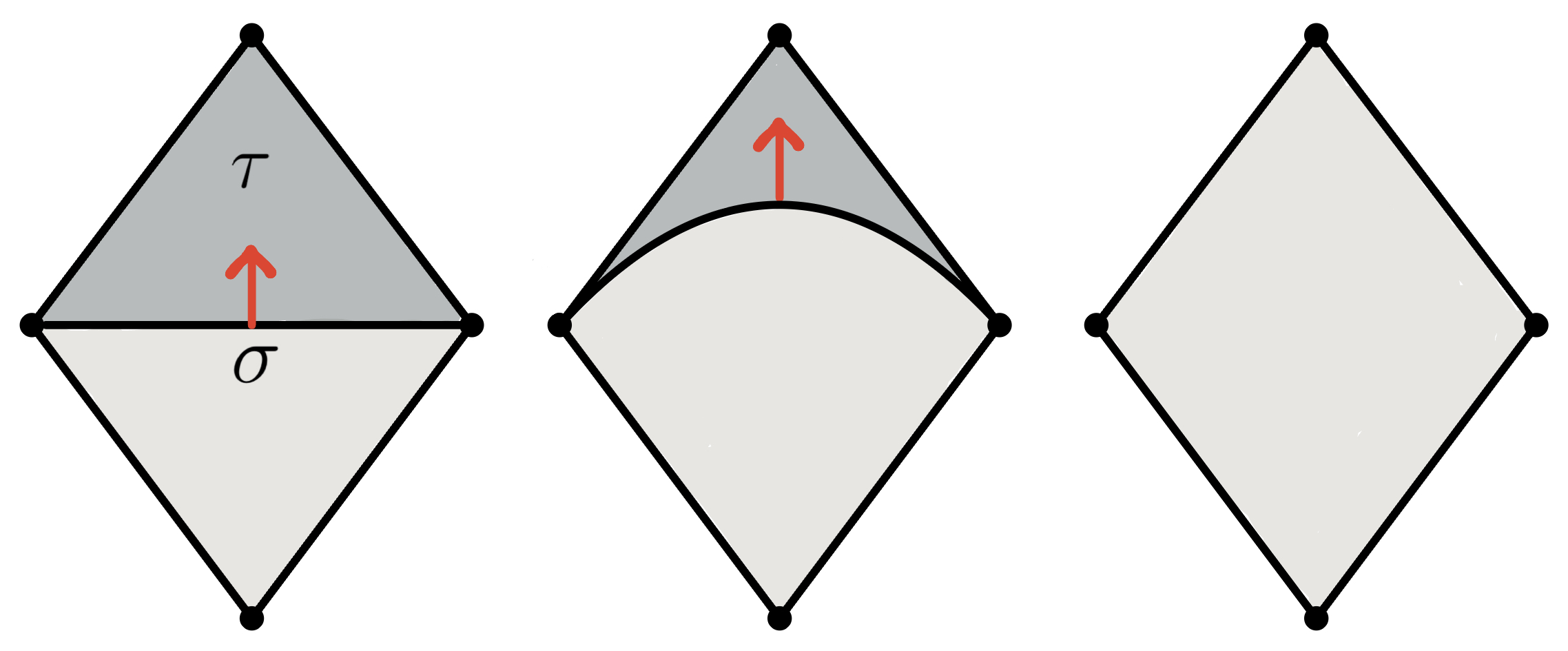}
\caption{Collapse (left) vs internal  collapse (right).}
\label{fig:collapse}
\end{figure}

Discrete Morse theory~\cite{forman_98, forman_02, kozlov_08}, developed by Forman in the 1990s as a combinatorial analogue of smooth Morse theory for manifolds, is a tool for studying the homotopy type of simplicial complexes via real-valued functions defined on them. Concretely, given a simplicial complex $K$ and a discrete Morse function $f \colon K \to \mathbb{R}$, there exists a reduced CW-complex built from the critical simplices of $f$, denoted by $\core_f(K)$, which is homotopy equivalent to $K$.
Moreover, a recent result establishes a connection between discrete Morse theory and simple homotopy theory: if $K$ has dimension $N$, then there exists a sequence of collapses and expansions from $K$ to $\core_f(K)$ through intermediate complexes of dimension at most $N+1$ (also called an $(N+1)$-deformation and denoted by $K \Nplusonedef \core_f(K)$)~\cite{fernandez_24}. This result relies on the fact that discrete Morse functions are equivalent to well-ordered sequences of internal collapses, which encode the deformation from $K$ to $\core_f(K)$~\cite{fernandez_24, kozlov_20}.
The construction of the reduced complex $\core_f(K)$, however, often sacrifices the original simplicial and combinatorial structure of $K$. The resulting attaching maps tend to be more intricate and are not explicitly determined, transferring the topological complexity of the space from many simplices with simple attaching maps to fewer cells (corresponding to the critical simplices) whose attaching maps are more complex.
This work focuses on a particular case of discrete Morse theory in which it is possible to provide an explicit and computable method for describing these attaching maps, even in higher dimensions.

In 2012, a strong version of homotopy theory was developed by Barmak and Minian~\cite{barmak_minian_12}. This approach builds upon simple homotopy theory by introducing  \textit{strong collapses} of simplicial complexes. Inspired by vertex reductions on posets~\cite{stong66}, strong collapses involve eliminating vertices $v$ (called \textit{dominated vertices}) whose link is a simplicial cone. In general, removing (the open star of) vertices with a contractible link preserves the homotopy type, as a consequence of the Gluing Theorem~\cite{brown_68}. Coned links, however, offer a systematic way to simple collapse $K$ to the subcomplex $K \smallsetminus v$ (see Figure~\ref{fig:strong_collapse}, left).

Strong collapses, as collections of simultaneous collapses, are highly efficient. Unlike Whitehead's collapses, strong collapses always yield the same minimal irreducible subcomplex --- the \textit{strong core} of $K$. However, the strong core of $K$ often coincides with $K$ itself, as many simplicial complexes lack dominated vertices. To overcome this limitation, one can consider performing strong collapses \textit{internally}, where domination is restricted to a subcomplex of $K$ (see Figure~\ref{fig:strong_collapse}, right). In such cases, a strategy must again be developed to recompute the attaching maps of the remaining cells, potentially compromising the simplicial structure of the reduced complex.

\begin{figure}[htb!]
\hspace{-5pt}\includegraphics[width=0.45\textwidth]{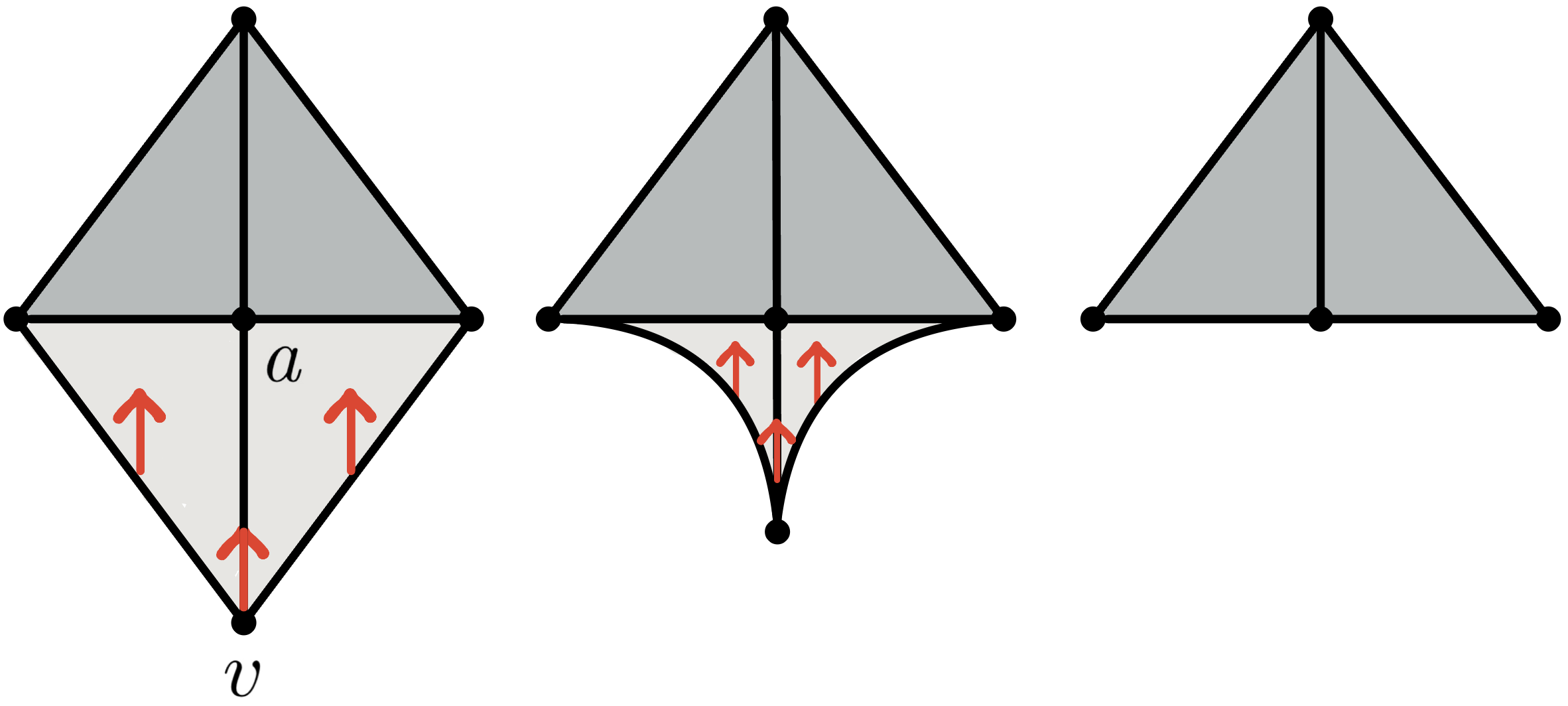}\hspace{40pt}\includegraphics[width=0.45\textwidth]{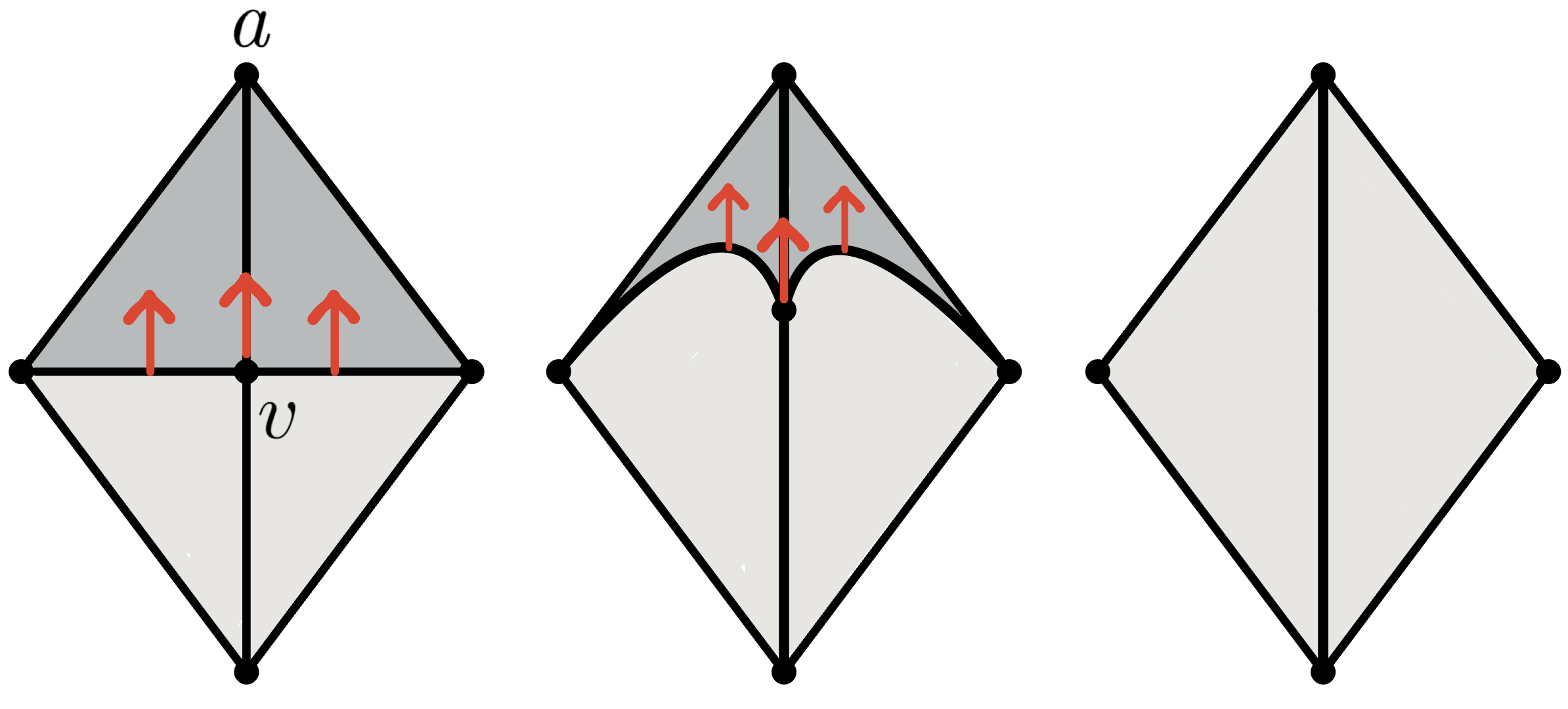}
\caption{Strong collapse (left) vs internal strong collapse (right).}
\label{fig:strong_collapse}
\end{figure}

In this work, we establish a method for studying \textit{internal strong collapses} and prove that the reduced CW-complex obtained from any sequence of such collapses is always \textit{regular} (i.e., its attaching maps are homeomorphisms onto their images). As a consequence, the CW-structure of the reduced complex is fully combinatorially determined by the incidences of its cells.
More broadly, we develop the theoretical foundations of \textit{strong Morse theory}, a version of discrete Morse theory in which discrete Morse functions correspond to sequences of internal strong collapses. We propose an efficient algorithm for computing the reduced complex $\core_f(K)$ for such discrete Morse functions $f$. To demonstrate its applicability, we use this method to identify more efficient regular models of the simplicial complexes in the Library of Triangulations~\cite{benedetti_24, Triangulation_Library, benedetti_14} by Benedetti and Lutz.

\subsection{Motivation and related work}
Classical discrete Morse theory does not provide an explicit description of the reduced homotopy equivalent CW-complex. A discrete Morse function on a simplicial complex $K$ induces a collection of critical simplices that correspond to the cells in the reduced CW-complex, along with information about their incidences, which is sufficient to determine the homology of the simplicial complex $K$. However, recovering the homotopy type of $K$ from the reduced CW-complex requires a more refined understanding of the attaching maps. Hence, except in special cases---such as when the homotopy type is fully determined by the list of critical cells (e.g., wedges of spheres)---the algorithmic reconstruction of homotopy types from discrete Morse data remains an open problem.

Some progress has been made toward addressing this challenge. In \cite{fernandez_24}, the author develops a combinatorial framework to describe the attaching maps of the reduced CW-complex for 2-dimensional simplicial complexes equipped with a discrete Morse function with a single critical 0-simplex. The approach makes strong use of the correspondence between 2-complexes and group presentations. However, it does not extend to higher dimensions, as the homotopy types of $N$-dimensional complexes with $N > 2$ generally cannot be fully captured combinatorially via group presentations. 
Rather than attempting to recover arbitrary attaching maps, the present work focuses on cases in which discrete Morse functions yield reduced complexes with particularly simple attaching maps---namely, homeomorphisms.

Our central goal  is to develop a combinatorial and computable framework for a strong discrete Morse theory that allows for the explicit reconstruction of the reduced CW-complex. By restricting Morse functions to those induced by internal strong collapses, we prove that the resulting reduced complex is always a regular CW-complex, and hence its homotopy type can be entirely determined by the poset of incidence cells. Moreover, as strong Morse theory is a restricted version of classical discrete Morse theory, the reduced complex $(N+1)$-deforms to the original simplicial complex $K$, where $N$ is the dimension of  $K$. Consequently, this theory also provides a computational method to perform $(N+1)$-dimensional deformations.

In \cite{fernandez_ternero_20}, the authors also establish connections between discrete Morse theory and strong collapses, focusing on general discrete Morse functions. Their work introduces a classification of internal collapses into critical and regular pairs, providing criteria to determine whether a sub-sequence of internal collapses is induced by a strong collapse. By contrast, our work only focuses on the specific class of Morse functions derived from only internal strong collapses.

Our work builds on the main principles of Bestvina–Brady Morse theory~\cite{bestvina_08, bestvinabrady_97}. In that framework, Morse functions are piecewise-linear (PL) functions defined on affine complexes, the critical points are the vertices, and changes in the topology of sublevel sets are studied through the \textit{descending links} of vertices.
Restricting to simplicial complexes and functions defined on their vertices, our theory refines this perspective by classifying vertices according to the structure of their descending links: either as \textit{descending dominated vertices} (their descending links are simplicial cones, and induce internal strong collapses) or as \textit{strong critical vertices}.
While Bestvina–Brady Morse theory studies local changes in the homotopy type of sublevel sets, our approach also aims at the combinatorial and global reconstruction of a reduced complex induced by these local changes. 

A complementary perspective is presented in \cite{nanda_19, nanda_18}, where the authors encode the homotopy type of a simplicial complex endowed with a discrete Morse function into a combinatorially defined \textit{flow category}. The classifying space of the flow category recovers the homotopy type of the original complex. However, these categories are often large and computationally challenging, limiting their application in algorithmic homotopy reconstruction.

\subsection{Main results and outline}

Let $K$ be a finite simplicial complex, and let $g \colon V(K) \to \mathbb{R}$ be a real-valued function on its set of vertices. The function $g$ induces a filtration of $K$ by subcomplexes, where each subcomplex consists of all simplices whose vertices lie below a given threshold in $g$. For each vertex $v \in V(K)$, we consider the first subcomplex in the filtration that contains $v$, and define its \emph{descending open star} as the collection of simplices in that subcomplex that contain $v$. A vertex is called \emph{strong critical} if, at the filtration step where it first appears, it is not dominated by any vertex from an earlier subcomplex.

This filtration of $K$ leads to a \textit{strong homotopy} analogue of the classical discrete Morse lemmas (Lemma \ref{lemma morse}). Intervals in the filtration without strong critical vertices correspond to strong collapses, while those with strong critical vertices introduce changes in the homotopy type, determined by the simplices in its descending open stars (Lemma \ref{lemma strong morse}). This reduction process produces a sequence of internal strong collapses, resulting in a reduced complex, called the \textit{strong internal core} and denoted by $\core_g(K)$. %, whose cells correspond to the critical simplices (with modified attaching maps). 
Notably, we prove that this reduced CW-complex is \textit{regular}, that is, its attaching maps are homeomorphisms with its image.

Our main result in strong Morse theory is as follows.

\begin{theorem*}[Strong Morse Theory] 
Let $K$ be a finite simplicial complex, and let $g \colon V(K) \to \mathbb{R}$ be a real-valued function on the vertex set of $K$. Then $K$ is homotopy equivalent to a \textbf{regular} CW-complex $\core_g(K)$, whose cells are in one-to-one correspondence with the simplices in the descending open stars of the strong critical vertices of $g$. Moreover, if $\dim(K) = N$, then $K \Nplusonedef \core_g(K)$.
\end{theorem*}

We also show that this result is a particular case of classical discrete Morse theory (Theorem~\ref{strong morse matching}). More precisely, for every function $g \colon V(K) \to \mathbb{R}$, there exists a discrete Morse function $f \colon K \to \mathbb{R}$ such that the critical simplices of $f$ are exactly the simplices in the descending open stars of the strong critical vertices of $g$. Moreover, the CW-complexes $\core_f(K)$ and $\core_g(K)$ coincide.\footnote{By a slight abuse of notation, we denote both by $\core$ depending on the context: $\core_f(K)$ denotes the internal core obtained from a discrete Morse function $f$ on simplices, and $\core_g(K)$ denotes the strong internal core obtained from a function $g$ on vertices.}

The combinatorial structure of the strong internal core admits a computational interpretation of internal strong collapses, which in turn leads to an efficient algorithm for their computation (Appendix \ref{appendix}, Algorithms \ref{alg:crit} and \ref{alg:internal_core_computation}). An implementation is publicly available at \url{https://github.com/ximenafernandez/Strong-Morse-Theory}. To illustrate the effectiveness of this method, we apply it to examples from the \emph{Library of Triangulations}~\cite{Triangulation_Library}, a repository of challenging simplicial complexes frequently used in the study of homotopy types. The results of the reductions are presented in Tables~\ref{tab:simplicial_complex_reduction_10it} and~\ref{tab:simplicial_complex_reduction_100it}.

Beyond its role in understanding homotopy types, the algorithmic reduction of simplicial complexes also has applications in topological data analysis, where the complexity of persistent homology computations increases rapidly with input size. For instance, computing degree-$d$ persistent homology of the Vietoris–Rips filtration of a point cloud of size $N$ has worst-case complexity $O(N^{3(d+2)})$~\cite{otter15}. The method introduced in~\cite{boissonnat_18} uses strong collapses to simplify Vietoris–Rips filtrations prior to computing persistent homology. Since our results extend algorithmic simplification of complexes via internal strong collapses, they have potential applications in further optimising such filtration-based computations.

\medskip

The remainder of this manuscript is organized as follows.
Section \ref{sec:discrete_morse} presents discrete Morse theory as a generalization of Whitehead's collapses.
In Section \ref{sec:strong morse}, we introduce a strong version of discrete Morse theory, based on strong collapses, which enables the recovery of a regular structure in the reduced CW-complex (the strong internal core).
Section \ref{sec:reconstruction core} provides a combinatorial construction of the strong internal core in terms of posets and acyclic matchings.
Finally, Section \ref{sec:random_reduction} discusses algorithms for the simplification of simplicial complexes and presents experiments on the Library of Triangulations \cite{Triangulation_Library}.
The Appendix details algorithms for constructing the critical poset and a random strong internal core.

\medskip 

\textbf{Notation.} We use $\simeq$ for homotopy equivalences, $\cong$ for homeomorphisms, and $|\cdot|$ for geometric realizations. The symbols $\sim$ and $\approx$ are reserved for relations. The relation $\prec$ denotes the face relation between simplices of consecutive dimension.

\section{Discrete Morse theory as generalized collapses}\label{sec:discrete_morse}
In this section, we revisit discrete Morse theory, framing it as a combinatorial generalization of  collapsibility of simplicial complexes, in the context of Whitehead's simple homotopy theory. The classical foundations of simple homotopy theory can be found in Whitehead's seminal works \cite{whitehead_39, whitehead_41, whitehead_50}, Milnor's influential article \cite{milnor_66}, and  Cohen's textbook \cite{cohen_73}.
Throughout this manuscript, we will restrict our attention to finite simplicial complexes. 

\begin{definition}
Let $K$ and $L$ be  simplicial complexes. There is an \textit{elementary collapse} from $K$ to $L$ if there exist simplices $\sigma, \tau \in K$ such that $L = K \smallsetminus \{\sigma, \tau\}$ and $\tau$ is the unique simplex containing $\sigma$ properly in $K$. In this case, $\sigma$ is called a \textit{free face} of $\tau$. 

A \textit{collapse}, denoted $K \co L$, is a sequence $K = K_n \ce K_{n-1} \ce \dots \ce K_0 = L$ of elementary collapses from $K$ to $L$. The inverse operation is called an \textit{expansion}, denoted $L \ex K$. If $K \co L$ we say that $L$ is a \textit{weak core} of $K$, and if $L$ has no free faces, we say it is \textit{minimal}.

There is an \textit{$N$-deformation}, denoted $K \Ndef L$, between $K$ and $L$ if there is a sequence $K = K_n, K_{n-1}, \dots, K_0 = L$ such that $K_i \co K_{i-1}$ or $K_i \ex K_{i-1}$, and $\dim(K_i) \leq N$ for all $i$. If there exists an $N$-deformation for some $N \in \mathbb{N}$ between $K$ and $L$, it is said that they have the same \textit{simple homotopy type}.
\end{definition}

Computing a minimal weak core of a simplicial complex provides a computational method for reducing its number of simplices while preserving its homotopy type. However, this approach is quite restrictive. For instance, many simplicial complexes lack free faces (e.g., triangulations of closed manifolds or famous examples such as the Dunce Hat \cite{zeeman_63} and Bing’s House \cite{bing_52}). Furthermore, a simplicial complex may have different (non-isomorphic) weak cores. For instance, if $D$ is any triangulation of the Dunce Hat, then $D \times I$ collapses to both $D$ and the singleton $*$. Thus, the order of collapses matters.
A more general approach is provided by $N$-deformations (e.g., the above shows that $D \threedef *$, even though its minimal weak core is $D$ itself). Nonetheless, the existence of an $N$-deformation between two simplicial complexes not computable, as this would imply being able to algorithmically determine the contractibility of spaces.

\medskip

Discrete Morse theory \cite{forman_02, forman_98}, inspired by its smooth counterpart \cite{milnor_63}, provides a more general combinatorial framework for simplifying the structure of a simplicial complex while preserving its homotopy type. Here, deformations are encoded in terms of discrete Morse functions. 

\begin{definition}\label{def:Morse_function}
Let $K$ be a simplicial complex. A function $f \colon K \to \mathbb{R}$ defined on the simplices of $K$ is called a \textit{discrete Morse function} if, for every simplex $\sigma \in K$, the set
\[
M(\sigma) = \{ \tau \prec \sigma \colon f(\tau) \geq f(\sigma) \}
\cup \{ \tau \succ \sigma \colon f(\tau) \leq f(\sigma) \}
\]
contains at most one element. A simplex $\sigma \in K$ is said to be \textit{critical} if $M(\sigma) = \varnothing$.
\end{definition}
%\textcolor{red}{It is known that any discrete Morse function is equivalent to an injective one, in the sense that the associated sets $M(\sigma)$ are identical for both (see the first paragraph of the proof of Theorem 3.3 in \cite{forman_98}).}

To every simplicial complex $K$, we associate its \textit{face poset} $\X(K)$, which is the poset of simplices of $K$ ordered by the face relation.
Each Morse function determines a pairing of simplices $M = \big\{\{\sigma, \tau\} \colon M(\sigma) = \{\tau\}\big\}$. It can be shown that $M$ is a \textit{matching} on the set of simplices of $K$. Moreover, the quotient of $\X(K)$ obtained by identifying matched simplices inherits a poset structure (see Section \ref{sec:reconstruction core} for more details). This matching is referred to as an \textit{acyclic matching} %, and we will denote the resulting poset by $\Loc_M(K)$ (see Section \ref{sec:reconstruction core} for more details). 
Chari \cite{chari_00} proved that acyclic matchings are in correspondence with discrete Morse functions.

The main theorem in discrete Morse theory establishes the connection between critical simplices of discrete Morse functions and the homotopy type of the simplicial complex.
\begin{theorem}[Forman  \cite{fernandez_24, forman_98, forman_02}] \label{discrete Morse theorem}
Let $K$ be a finite simplicial complex of dimension $N$, and let $f\colon K \to \RR$ be a discrete Morse function. Then, there is an $(N+1)$-deformation from $K$ to a CW-complex with exactly one cell of dimension $d$ for every critical $d$-simplex of $f$. 
\end{theorem}
We denote by $\core_f(K)$ the CW-complex obtained from the critical simplices of $f$ in Theorem~\ref{discrete Morse theorem}, and refer to it as the \emph{internal core} of $K$ associated to $f$. 
It is important to note that Theorem~\ref{discrete Morse theorem} does not provide an explicit description of how the critical cells are attached along their boundaries in $\core_f(K)$. In \cite[Prop. 1.3, Thm. 1.6]{fernandez_24}, the author offers a recursive construction of the attaching maps for the critical cells, based on the finiteness of $K$; however, the construction is only made explicit for 2-dimensional cells.

\medskip

The deformation from $K$ to $\core_f(K)$ is described locally by the Morse Lemma \ref{lemma morse}. Any discrete Morse function $f \colon K \to \mathbb{R}$ induces a filtration of $K$ by sublevel subcomplexes: for each $\alpha \in \mathbb{R}$,  $K_\alpha$ is the smallest subcomplex of $K$ containing all simplices $\sigma$ with $f(\sigma) \leq \alpha$.
The following result describes the evolution of the homotopy type complex along the filtration: intervals that do not contain critical simplices correspond to collapses, while those containing a single critical simplex account for homotopy changes. 

\begin{lemma}[Thm. 3.3 and 3.4, \cite{forman_98}]\label{lemma morse}
Let $f \colon K \to \mathbb{R}$ be a discrete Morse function.
\begin{enumerate}[A.]
\item If $f^{-1}\big((\alpha, \beta]\big)$ contains no critical simplices, then $K_\beta \co K_\alpha$.
\item If $f^{-1}\big((\alpha, \beta]\big)$ contains exactly one critical $k$-simplex, then $K_\beta \simeq K_\alpha \cup e^k$, where $e^k$ is a $k$-dimensional cell.
\end{enumerate}
\end{lemma}

From a combinatorial perspective, discrete Morse functions can be viewed as a totally ordered list of internal collapses.
\begin{proposition}[Internal collapse {\cite{fernandez_24}}] \label{internal collapse}
Let $K$ be a finite simplicial complex of dimension $N$. If a subcomplex $K_0 \subseteq K$ collapses to a subcomplex $L_0$, then $K \Nplusonedef L$, where $L$ is a CW-complex obtained from $L_0$ by attaching one $k$-cell for each $k$-simplex in $K \smallsetminus K_0$.
\end{proposition}

Proposition~\ref{internal collapse} also admits a formulation in terms of CW-complexes; see \cite[Proposition~1.3]{fernandez_24}.
We refer to the deformation from $K$ to $L$ as an \emph{internal collapse}. The collapse takes place within a subcomplex $K_0 \subseteq K$, while the resulting deformation affects the entire complex through the reattachment of cells corresponding to the simplices in $K \smallsetminus K_0$.

Internal collapses arise naturally in discrete Morse theory: for any discrete Morse function $f \colon K \to \mathbb{R}$, the transitions between sublevel complexes as described in the Morse Lemma (Lemma~\ref{lemma morse}) correspond to internal collapses. Specifically, if the interval $(\alpha, \beta]$ contains no critical simplices, then $K_\beta \simeq K_\alpha$, and there is an internal collapse from $K$ to a CW-complex obtained from $K_\alpha$ by attaching a $k$-cell for each $k$-simplex in $K \smallsetminus K_\beta$. In general, the entire deformation from $K$ to $\core_f(K)$ can be interpreted as a well-ordered sequence of internal collapses.

We summarize the viewpoint of discrete Morse theory as a generalization of classical collapses in the following statement.

\begin{proposition}\label{internal core vs weak core}
Let $K$ be a finite simplicial complex.
\begin{enumerate}
    \item If $K \searrow L$, then there exists a discrete Morse function $f \colon K \to \mathbb{R}$ such that $\core_f(K) = L$.
    \item If $f \colon K \to \mathbb{R}$ is a discrete Morse function, then there exists a sequence of internal collapses from $K$ to $\core_f(K)$. 
\end{enumerate}
\end{proposition}

\iffalse
\begin{example}
  \textcolor{red}{weak core and internal core of a simplicial complex}  
\end{example}
\fi

%complexity of discrete morse functions
%flow lines vs gradient vector fields

%The unstable manifold Wu(p) of a critical point p then simply consists of all flow lines x(t) with $x(−\infty) = p$, i.e. of those flow lines that emanate from p. the unstable manifold Wu(p) is topologically a cell (i.e. homeomorphic to an open ball) of dimension $\mu(p)$, and the manifold $\M$ is the union of the unstable manifolds of the critical points of the function.

\section{Strong discrete Morse theory}\label{sec:strong morse}

Building on the strong homotopy theory for simplicial complexes introduced by Barmak and Minian \cite{barmak_minian_12}, we present a  version of discrete Morse theory centred on the concept of \textit{internal strong collapses}. We recall first the notion of strong collapses.

Given a simplicial complex $K$ and a vertex $v \in V(K)$, the \emph{open star} of $v$, denoted $\st(v, K)$, is the collection of simplices in $K$ that contain $v$. The subcomplex of simplices disjoint from $v$ is denoted $K \smallsetminus v$. The \emph{link} of $v$ in $K$, denoted $\lk(v, K)$, is the subcomplex of $K \smallsetminus v$ consisting of simplices $\sigma$ such that $\sigma \cup \{v\}$ is a simplex of $K$.
A \emph{simplicial cone} over $K$ with apex $a$ (a vertex not in $K$) is the complex $a*K$ consisting of all simplices of $K$, the vertex $\{a\}$, and all simplices of the form $\sigma \cup \{a\}$ with $\sigma \in K$. 
The \emph{closed star} of $v$, denoted $\overline{\st}(v, K)$, is the subcomplex $v \ast \lk(v,K)$.

\begin{definition}[Def.~2.1, \cite{barmak_minian_12}]
Let $K$ be a simplicial complex. A vertex $v$ of $K$ is \emph{dominated} by a vertex $a$ if $\lk(v,K)$ is a simplicial cone $a*K_0$, where $K_0$ is a subcomplex of $K$. In that case, there is an \emph{elementary strong collapse} from $K$ to $K \smallsetminus v$, denoted $K \searrow K \smallsetminus v$. A \emph{strong collapse} $K \sco L$ from a simplicial complex $K$ to a subcomplex $L$ is a sequence of elementary strong collapses starting at $K$ and ending at $L$. In that case, we say that $L$ is a \emph{strong core} of $K$. If $L$ has no dominated vertices, we say that $L$ is a \emph{minimal strong core} of $K$.\footnote{This is simply called a \emph{core} in \cite{barmak_minian_12}. We will use the term \emph{minimal strong core} in this manuscript to avoid confusion with other core constructions.}
\end{definition}

If $v \in K$ is dominated by $a$, then there is an explicit strong deformation retraction $|r_a|\colon |K| \to |K \smallsetminus v|$ 
induced by the simplicial map
\begin{equation}\label{eq:retraction collapse}
r_a(\sigma) =
\begin{cases}
\sigma & \text{if } v \notin \sigma, \\
\{a\} \cup \sigma \smallsetminus \{v\} & \text{if } v \in \sigma.
\end{cases}
\end{equation}
More generally, if $K \sco L$ via a sequence of elementary strong collapses
\[
K = K_0 \esc K_1 \esc \dots \esc K_n = L,
\]
with $K_{i+1} = K_i \smallsetminus v_i$ and $v_i$ dominated by $a_i \in K_i$, then there is a strong deformation retraction $|r|\colon |K| \to |L|$ induced  by the composition
\begin{equation} \label{eq:composition retraction collapse}
r := r_{a_n} \circ r_{a_{n-1}} \circ \dots \circ r_{a_1}.
\end{equation}
Moreover, if $K \sco L$, then $K \co L$ \cite[Rmk.~2.4]{barmak_minian_12}. That is, Whitehead's notion of collapse is weaker than the notion of strong collapse.

\begin{theorem} [Thm. 2.11. \cite{barmak_minian_12}] Any sequence of strong collapses from a simplicial complex $K$ to a subcomplex without dominated vertices yields  a unique (up to isomorphism) subcomplex of $K$. That is, the minimal strong core of a simplicial complex is unique up to isomorphism.
\end{theorem}

Recall that a CW-complex $K$ is \emph{regular} if the attaching map $\varphi \colon S^{n-1} \to K^{(n-1)}$ of each open $n$-cell $e^n$ is a homeomorphism onto its image $\partial e^n$ in the $(n-1)$-skeleton. A regular CW-complex is said to be \emph{combinatorial} if, for each cell $e^n$, the attaching map --- regarded as a cellular map from $S^{n-1}$ endowed with a CW-structure --- is a \emph{combinatorial map}, meaning that it maps each open $k$-cell of $S^{n-1}$ homeomorphically onto an open $k$-cell of $K$ (see \cite[Ch.~II]{hog-angeloni_metzler_93}).
The geometric realization $|K|$ of any simplicial complex $K$ is canonically a regular combinatorial CW-complex.
Given a vertex $v$ in a regular CW-complex $K$, we define the \emph{open star} $\st(v,K)$ as the union of all open cells in $K$ that contain $v$ as a face. We denote by $\overline{e}$ the closure of a cell $e \in K$, i.e., the smallest subcomplex of $K$ containing $e$, which consists of $e$ together with all its faces.
Similarly, we denote by $K \smallsetminus v$ the subcomplex of $K$ consisting of all cells $e$ such that $v \notin \overline{e}$.

\medskip

The following result formalizes the notion of an \emph{internal strong collapse}, illustrated in Figure~\ref{fig:strong_collapse}.

\begin{theorem} \label{theo1:  strong internal collapse}  
Let $K$ be a regular combinatorial CW-complex, and let $w$ be a vertex of $K$ such that the subcomplex $K \smallsetminus w$ is a simplicial complex. Suppose there exists a subcomplex $L \subseteq K \smallsetminus w$ such that $K \smallsetminus w \sco L$ via a sequence of strong collapses.

Then there exists a regular combinatorial CW-complex
\[
Z = L \cup \bigcup_{\ell} \tilde{e}_\ell
\]
that is homotopy equivalent to $K$, where:
\begin{itemize}
    \item $L$ is embedded as a subcomplex of $Z$,
    \item for each cell $e_\ell \in K$ such that $w \in \overline{e_\ell}$, there is a corresponding $\ell$-cell $\tilde{e}_\ell$ in $Z$ not contained in $L$, and
    \item the attaching map of each $\tilde{e}_\ell$ is induced from that of $e_\ell$ and the strong deformation retraction $|r| \colon |K \smallsetminus w| \to |L|$ determined by the strong collapse.
\end{itemize}

Moreover, if $\dim(K) = N$, then $K \Nplusonedef Z$.
\end{theorem}

\begin{proof}
Consider the pushout diagram
\[
\xymatrix{
|K \smallsetminus w| \ar[r]^{|r|} \ar@{^{(}->}[d]_{i} & |L| \ar[d]^{\bar{i}}\\
K \ar[r]^{\bar{r}} & Z
}
\]
where $i$ is the inclusion map and $|r| \colon |K \smallsetminus w| \to |L|$ is the strong deformation retraction induced by the strong collapse. Since $|r|$ is a homotopy equivalence and $i$ is a closed cofibration, the Gluing Theorem (see, e.g., \cite[7.5.7, Corollary 2]{brown_68}) implies that $\bar{r}$ is also a homotopy equivalence.

We now define a regular combinatorial CW-structure on $Z$. Assume first that the collapse $K \smallsetminus w \sco L$ is elementary, i.e., $L = (K \smallsetminus w) \smallsetminus v$ for some vertex $v$ dominated by $a \in K \smallsetminus w$.

The $0$-skeleton of $Z$ is defined as $Z^{(0)} = L^{(0)} \cup \{w\}$. We define a map $q \colon K^{(0)} \to Z^{(0)}$ by
\[
q(x) = 
\begin{cases}
x & \text{if } x \in L^{(0)} \cup \{w\},\\
a & \text{if } x = v,
\end{cases}
\]
which agrees with the retraction $|r|$ on $K^{(0)} \smallsetminus \{w\}$.

Inductively, suppose that a regular combinatorial CW-structure has been defined on $Z^{(\ell-1)}$, and that the map $q \colon K^{(\ell -1)} \to Z^{(\ell-1)}$ is combinatorial and coincides with the simplicial retraction $|r|$ on $(K \smallsetminus w)^{(\ell-1)}$. Let $e$ be an $\ell$-cell in $K$ with attaching map $\varphi\colon S^{\ell-1} \to K^{(\ell-1)}$.

If $w \notin \overline{e}$, then $e$ belongs entirely to the subcomplex $K \smallsetminus w$, and its image in $Z$ is defined via $q = |r|$. Since $r \colon K \smallsetminus w \to L$ is simplicial, and the attaching maps in $K\smallsetminus w$ are simplicial, it follows that $q(e)$ inherits a regular combinatorial CW-structure. In particular, if $e \in L$, then $q$ acts as the identity on $e$.

If $w \in \overline{e}$, we distinguish three cases:

\begin{enumerate}
\item If $v \notin \overline{e}$, then $q$ acts as the identity on $e$, and the attaching map of $e$ remains unchanged in $Z$,  inheriting the regular combinatorial CW-structure from $K$.

\item If $v \in \overline{e}$ but $a \notin \overline{e}$, then $r$ sends $v$ to $a$ in $\partial e \cap (K\smallsetminus w)$, resulting in a relabelling, inducing a bijective identification on the boundary. The  attaching map $\tilde{\varphi} \colon S^{\ell - 1} \to Z^{(\ell - 1)}$ of $\tilde e\in Z$ is given by $\tilde{\varphi} := r \circ \varphi$, and the image cell $q(e)$ is obtained from $e$ via this relabelling.

\item If both $v$ and $a$ belong to $\overline{e}$, then the retraction $r$ modifies the attaching map non-trivially. Since $K$ is regular and combinatorial, the boundary $\partial e$ is a PL $(\ell -1)$-sphere. Let 
\[
A := |\overline{\st}(v, K \smallsetminus w)| \cap \partial e,\quad B := |\lk(v, K \smallsetminus w)| \cap \partial e.
\]
Then $A$ is a PL $(\ell -1)$-ball and $B$ a PL $(\ell-2)$-ball, and the strong collapse induces a retraction $r|_A \colon A \to B$.
By \cite[Ch.~3]{rourke_sanderson_72}, there exist regular neighbourhoods $U \supseteq A$ and $V \supseteq B$ in $\partial e$ and a map $f \colon \partial e \to \partial e$ such that  $ f|_A = r|_A $, and $ f|_{\partial e \setminus A} \colon \partial e \setminus A \to \partial e \setminus B $ is a homeomorphism.
This induces a homeomorphism
\[
\bar{f} \colon \partial e / \!\sim\  \longrightarrow \partial e, \quad x \sim r(x)\ \text{for } x \in A.
\]
If $ \varphi \colon S^{\ell -1} \to \partial e $ is the original attaching map of $e$, then the modified attaching map is given by:
\[
\xymatrix@C=4em{
S^{\ell -1} \ar[r]^{\varphi} \ar@{->>}[d] & \partial e \ar[d]^q \\
S^{\ell -1}/\! \approx \ar[r]^{\tilde{\varphi}} & q(\partial e) \subseteq Z^{(\ell -1)}
}
\]
where $ x \approx y $ if $ \varphi(x) \sim \varphi(y) $ under the identification induced by $r$. Since $ \bar{f} $ is a homeomorphism, the quotient space $ S^{\ell -1}/\!\approx $ is homeomorphic to $ S^{\ell -1} $, and the new attaching map $ \tilde{\varphi} $ is a homeomorphism onto its image.
Therefore, the cell $ \tilde{e}$ is attached regularly, and inherits a combinatorial CW-structure. Define $q(e)= \tilde e$. \end{enumerate}

The general case follows by induction on the number of elementary strong collapses in the sequence 
\[
K \smallsetminus w = K_0 \sco K_1 \sco \cdots \sco K_n = L.
\]
At each step, the same argument applies to the retraction $ K_i \to K_{i+1} $, ensuring that the regularity and combinatorial structure are preserved. Since $w$ is not removed in the process, the identification maps never collapse an entire boundary sphere to a point.

Thus, $Z$ is a regular combinatorial CW-complex homotopy equivalent to $K$. The fact that this is an $(N+1)$-deformation follows from \cite[Prop.~1.3]{fernandez_24}.
\end{proof}

The previous result generalizes to the case where $K_0 \sco L_0$ and $K_0$ is obtained from a simplicial complex $K$ by successively removing vertices (together with their open stars). In this case, $K_0$ is a \emph{full subcomplex} of $K$, meaning that every simplex of $K$ whose vertices all belong to $V(K_0)$ is itself a simplex in $K_0$.

\begin{theorem}[Strong internal collapse] \label{theo2: strong internal collapse}
Let $K$ be a simplicial complex of dimension $N$, and let $K_0 \subseteq K$ be a full subcomplex. Suppose $K_0$ strong collapses to a subcomplex $L_0$. Then $K$ is $(N+1)$-homotopy equivalent to a combinatorial regular CW-complex $L$ obtained from $L_0$ by attaching one $k$-cell for each $k$-simplex in $K \smallsetminus K_0$.
\end{theorem}

\begin{proof}
Follows by induction on the vertices in $V(K)\smallsetminus V(K_0)$, applying Theorem \ref{theo1: strong internal collapse} at each step.
\end{proof}

Under the assumptions of Theorem~\ref{theo2: strong internal collapse}, we say that there is a \emph{strong internal collapse} from $K$ to $L$.

\medskip

Let $K$ be a finite simplicial complex, and let $g \colon V(K) \to \mathbb{R}$ be a real-valued function on its vertex set. The function $f$ induces a filtration of $K$ by sublevel subcomplexes $\{K_\alpha\}_ {\alpha\in \RR}$, where each $K_\alpha$ is the subcomplex spanned by the vertices in the sublevel set $g^{-1}\big((-\infty, \alpha]\big)$.

\begin{definition}
Let $g \colon V(K) \to \mathbb{R}$ be a real-valued function, and let $v \in V(K)$ be a vertex.
The \textit{descending open star} of $v$, denoted $\st^\downarrow(v, K)$, is the open star of $v$ in the subcomplex $K_{g(v)}$. Similarly, the \textit{descending link} of $v$, denoted $\lk^\downarrow(v, K)$, is the link of $v$ in $K_{g(v)}$.
We say that $v$ is \textit{descending dominated} if it is dominated in $K_{g(v)}$ by a vertex $a$ with $g(a) < g(v)$. Otherwise, $v$ is called a \textit{strong critical vertex}.
\end{definition}

\iffalse
\begin{remark} In the particular case when $K$ is the order complex $\K(X)$ of a finite poset  $(X, \leq)$,  $K$ inherits an order structure from $X$ (that is, $(K, \leq)$ is an ordered simplicial complex). Moreover, $K_{< a} = \K(\hat U_a^X)$ and  $v\in K$ is  a descending dominated vertex if it is a \textit{down beat point} in $X_{\leq  v} $.
\end{remark}
\fi

We have an analogous version of the classical discrete Morse Lemma \ref{lemma morse} in this setting.

\begin{lemma}\label{lemma strong morse}
Let $K$ be a simplicial complex and $g \colon V(K) \to \mathbb{R}$ a real-valued function.
\begin{enumerate}[A.]
    \item If $g^{-1}\big((\alpha, \beta]\big)$ contains no strong critical vertices, then $K_\beta \sco K_\alpha$.
    
    \item If $g^{-1}\big((\alpha, \beta]\big)$ contains a single strong critical vertex $w$, then $K_\beta$ is homotopy equivalent to a CW-complex of the form
    \[
    K_\alpha \cup \bigcup_i e_i,
    \]
    where the cells $e_i$ are in one-to-one correspondence with the simplices in $\st^{\downarrow}(w, K)$. This correspondence preserves dimensions, and the attaching maps of the cells are determined as in Theorem \ref{theo1: strong internal collapse}, resulting in a regular CW-complex structure.
\end{enumerate}
\end{lemma}

\begin{proof}
The function $g$ induces a total order $\{v_1, v_2, \dots, v_n\}$ on $V(K)$ such that $g(v_i) \leq g(v_j)$ whenever $i < j$. For each $i$, let $K_i$ be the subcomplex of $K$ generated by all simplices whose vertices lie in $\{v_1, v_2, \dots, v_i\}$. %We have $g^{-1}((\alpha, \beta]) = \{v_k, v_{k+1}, \dots, v_\ell\}$ for some $k \leq \ell$.

For part~A., assume that $g^{-1}\big((\alpha, \beta]\big)$ contains no strong critical vertices. Then for each $v_i \in g^{-1}\big((\alpha, \beta]\big)$, the descending link $\lk^{\downarrow}(v_i, K) = \lk(v_i, K_{g(v_i)})$ is a simplicial cone with apex $a_i$, where $g(a_i) < g(v_i)$. Since $K_i \subseteq K_{g(v_i)}$, we have
\[
\lk(v_i, K_i) = \lk(v_i, K_{g(v_i)}) \smallsetminus \{\sigma \in \lk(v_i, K_{g(v_i)}) \colon \sigma \ni v_j \text{ with } j > i \text{ and } g(v_j) = g(v_i)\}.
\]
This subcomplex remains a cone with apex $a_i$. Therefore, $v_i$ is dominated in $K_i$ by $a_i$, and $K_i \esc K_i \smallsetminus v_i = K_{i-1}$. Applying this inductively, we obtain a sequence of strong collapses from $K_\beta$ to $K_\alpha$.

Part~B. follows from part~A. and Theorem~\ref{theo1: strong internal collapse}.
\end{proof}

%By Lemma \ref{lemma 1 morse no critical}, $K_{\leq b} \sco K_{\leq f(w)}$ and $K_{< f(w)} \sco K_{\leq a}$. Now, by Lemma \ref{strong internal collapse}, $K_{\leq f(w)} = K_{< f(w)} \cup w \simeq K_{\leq a} \cup \bigcup_i e_i$, with $e_i$ in correspondence to the simplices in $\st^\downarrow(w, K_{\leq f(w)})$.
%The attaching map of the cells associated to a $k$-simplex $\sigma\in \st^\downarrow(w, K_{\leq f(w)})$ for some critical vertex $w$ is obtained by identifying the dominated vertices in $\sigma$ with domination vertices.

Building on Theorem \ref{theo2: strong internal collapse} and Lemma \ref{lemma strong morse}, we establish the following result, which formalize the relationship between internal strong collapses, regular CW-complexes, and discrete Morse theory. 

\begin{theorem}[Strong discrete Morse theory]\label{strong Morse theorem}  
Let $K$ be a finite simplicial complex, and let $g \colon V(K) \to \mathbb{R}$ be a real-valued function. Then, $K$ is homotopy equivalent to a \textit{regular} CW-complex $\core_g(K)$, whose  cells are in one-to-one correspondence with the simplices in the descending open star of the strong critical vertices of $g$.  
\end{theorem}

\begin{proof}
Subdivide the image of $g$ into subintervals, each satisfying the conditions of Lemma \ref{lemma strong morse} A. or B., and iteratively apply the corresponding deformations to $K$. 
The cells in the resulting CW-complex correspond to the simplices in $\{\st^\downarrow(w, K) \colon w \text{ is a critical vertex in } K\}$. 
%This process is guaranteed to terminate, as the value of $g$ strictly decreases when moving from a dominated vertex to its dominating vertex.  
\end{proof}

We denote by $\core_g(K)$ the regular CW-complex obtained from the critical simplices of $g \colon V(K) \to \mathbb{R}$ in Theorem~\ref{strong Morse theorem}, and refer to it as the \emph{strong internal core} of $K$ associated to $g$. 
Notice that the strong internal core depends on a choice of dominating vertices for the strong internal collapses, but we omit this information unless necessary. 

\begin{proposition}[Connection to strong collapses]\label{prop:strong-vs-stronginternal}
If $K \sco L$, then there exists a function $g \colon V(K) \to \mathbb{R}$ such that $\core_g(K) = L$.
\end{proposition}

\begin{proof}
Let $v_1, v_2, \dots, v_n \in V(K)$ be the sequence of dominated vertices corresponding to a strong collapse from $K$ to $L$, i.e., define $K_n = K$ and $K_{i-1} = K_i \smallsetminus v_i$, with each $K_i \esc K_{i-1}$ and $K_0 = L$. Define a function $g \colon V(K) \to \mathbb{R}$ by
\[
g(v) = 
\begin{cases}
    i & \text{if } v = v_i, \\
    0 & \text{if } v \in V(L).
\end{cases}
\]

We claim that the strong internal core, $\core_g(K)$, coincides with $L$. By construction, each $v_i$ is dominated in $K_i$ by some vertex $a_i \in K_i$. Since $v_i$ is removed at stage $i$, and $a_i \in K_j$ for some $j < i$, it follows that $g(a_i) < g(v_i)$. Therefore, $v_i$ is descending dominated. On the other hand, any vertex $v \in V(L)$ satisfies $g(v) = 0$, and hence $g^{-1}((-\infty, g(v))) = \varnothing$. Thus, $v$ cannot be descending dominated and is a strong critical vertex.

Hence, the only strong critical vertices of $g$ are those in $L$, and the strong internal core, $\core_g(K)$, coincides with  $L$.
\end{proof}

\begin{theorem}[Connection to discrete Morse theory]\label{strong morse matching} Let $K$ be a finite simplicial complex and let $g\colon V(K) \to \RR$ be a real-valued map. Then, there exists a discrete Morse function $f\colon K\to \RR$ on $K$ such that its critical simplices are the simplices in $\{\st^\downarrow(w, K) : w \text{ is a strong critical vertex of } g\}$. Moreover,  $\core_{g}(K)\cong  \core_{f}(K)$, that is the strong internal core associated to  $g$ is precisely the internal core  associated to $f$.\end{theorem}

\iffalse
Analogously for discrete Morse functions.
Let $0<\epsilon<1/\dim(K)$ and assume $f\colon K^{(0)}\to \ZZ$.
Define 
$\bar f\colon K\to \RR$ as follows. For every simplex $\sigma\in K$, determine $v_\sigma = \arg \max\{f(v)\colon  v\in \sigma \text{ descending dominated}\}$.   In that case, define $a_\sigma$ dominating vertex associated to $v_\sigma$. If $\sigma$ only has critical vertices, define $v_\sigma = a_\sigma = \varnothing$. Define
\[
\bar f(\sigma) = \begin{cases}\max f(\sigma)+\dim(\sigma \smallsetminus a_\sigma)\epsilon & \text{if }\sigma \text{ contains dominated vertex}\\ \max f(\sigma)+\dim(\sigma)\epsilon & \text{otherwise.}
\end{cases}
\]

With Morse functions: 
$\bar f:K\to \RR$ defined as:

$\bar f(v) = f(v)$ for all $v\in K^{(0)}$

Let $0<\epsilon<1/\dim(K)$. For every simplex $\sigma\in K_{\leq f(v_i)}\smallsetminus K_{<f(v_{i-1})}$,

if $v_i$ dominated by $a_i$, then for every pair $\sigma, a\sigma$ of simplices in $\lk(v_i, K_{\leq f(v_i)}$ ($a\notin \sigma$), 

$\bar f(\sigma) = \bar f(a\sigma) = f(v_i)+\dim(\sigma)\epsilon$.

if $v_i$ is not dominated, then for every simplex $\sigma \ni v_i$, 
$\bar f(\sigma) = f(v_i)+\dim(\sigma)\epsilon$.

In summary, it is always $\bar f(\sigma) = f(v_i)+\dim(\sigma \smallsetminus a)\epsilon$, if $a$ dominating vertex exists; otherwise it is just $\bar f(\sigma) = f(v_i)+\dim(\sigma)\epsilon$.
\fi

\begin{proof} Let $C = \{\st^\downarrow(w, K) : w \text{ is a strong critical vertex of } g\}\subseteq K$.
Given the correspondence between discrete Morse functions and acyclic matchings  \cite{chari_00}, 
we will construct an acyclic matching $M$ with critical simplices $C$ as follows. For each simplex $\sigma \in K$ containing descending dominated vertices, let  
\[
v_\sigma = \arg \max\{g(v) \colon v \in \sigma \text{ is a descending dominated vertex}\},
\]
and let $a_\sigma$ denote the vertex dominating $v_\sigma$. The pairing $M$ is then defined as follows: if $a_\sigma \in \sigma$, we include the pair $(\sigma \setminus \{a_\sigma\}, \sigma)$ in $M$; otherwise, we include the pair $(\sigma, \sigma \cup \{a_\sigma\})$ in $M$.

To verify that $M$ is a matching, assume for contradiction that a simplex in $K$ is matched to more than one simplex. Without loss of generality, suppose $(\sigma, \{a\}\cup \sigma) \in M$ with $v_\sigma$ the associated descending dominated vertex. If there exists another pair $(\{a\}\cup\sigma, \{a',a\}\cup\sigma) \in M$, then $a'$ dominates a vertex $v' \in \sigma$. Since $v'$ must also satisfy $v' = \arg\max\{g(v) \colon v \in \sigma$ is descending dominated$\}$
 this contradicts the uniqueness of $v_\sigma$. Alternatively, if there exists a pair $(\sigma', \{a\}\cup\sigma) \in M$, then $a\sigma = \{a'\}\cup\sigma'$ implies $\sigma = \{a'\}\cup\tau$ and $\sigma' = \{a\}\cup\tau$ for some simplex $\tau$. This forces $\tau$ to contain two distinct dominated vertices, $v_\sigma$ and $v'$, each maximizing $g(v)$, which is again a contradiction. Therefore, $M$ is a valid matching.

To establish that $M$ is acyclic, consider any pair $(\sigma, \{a\}\cup\sigma) \in M$. For any $\tau \succ \sigma$, we claim that $(\tau \smallsetminus \{a'\}, \tau) \notin M$ for any $a' \neq a$. Since $\tau \succ \sigma$, the simplex $\tau$ contains $v_\sigma$. If $v_\sigma$ also satisfies $v_\sigma = \arg \max\{g(v) \colon v \in \tau \text{ is a descending dominated vertex}\}$, then $(\tau, \{a\}\cup\tau) \in M$, as $a$ dominates $v_\sigma$. Alternatively, if there exists another dominated vertex $w \in \tau$ such that $w > v_\sigma$, then $(\tau, \{a_w\}\cup \tau) \in M$, where $a_w$ dominates $w$. In either case, $(\tau \smallsetminus \{a'\}, \tau) \notin M$ for any $a' \neq a$, ensuring that $M$ is acyclic.
This completes the proof.
\end{proof}

%Let $\bar f \colon K \to \RR$ be a Morse function associated with $M$, such that $f|_{K^{(0)}} = f$. Notice that for every critical $k$-simplex $\tau$ of $\bar f$ containing a critical vertex $w$ of $f$, and for every dominated vertex $v_1 \leq v_2 \leq \cdots \leq v_l$ in $\tau$ in $K_{\bar f(\tau)}$, the new attaching map $\varphi_\tau$ of $|\tau|$ is the composition of $r_1 \circ r_2 \circ \cdots \circ r_l \circ \varphi_\tau \colon D^n \to K_{\bar f(\tau)} \smallsetminus \{v_1, v_2, \dots, v_l\}$.

%a Morse function $\bar{f} \colon K \to \RR$ such that $\bar{f}|_{K^{(0)}} = f$ and the critical simplices of $\bar{f}$ are the simplices in $\{\st^\downarrow(w, K) : w \text{ is a critical vertex of } f\}$. Moreover, $\core_f(K)$ is the reduced Morse complex of $K$ with respect to $\bar{f}$. In particular, if $\dim(K) = n$, then $K \nplusonedef \core_f(K)$.

\begin{remark}[Connection to Bestvina--Brady Morse theory]\label{Bestvina-Brady Morse}
Let $K$ be a finite simplicial complex. A Bestvina--Brady Morse function $h \colon |K| \to \mathbb{R}$ is given by an injective function $g \colon V(K) \to \mathbb{R}$ that extends linearly over simplices. In this setting, the critical points of $h$ are precisely the vertices of $K$, and the changes in homotopy type of sublevel sets occur at critical levels by attaching cones over the descending links of the corresponding critical vertices.
In our framework, we also consider a vertex function $g$, but classify each vertex as either \emph{strong critical} or \emph{descending dominated}, depending on whether its descending link forms a simplicial cone. This distinction not only describes the local homotopy behaviour (descending dominated vertices have coned descending links and, hence, do not change the homotopy type of sublevel complexes), but crucially allows a global, purely combinatorial reconstruction of the reduced complex $\core_g(K)$.

Forman's discrete Morse theory fits within the Bestvina--Brady framework, as observed by Zaremsky~\cite{zaremsky_22}, and can also be interpreted through the lens of strong discrete Morse theory. Specifically, any injective\footnote{Any discrete Morse function can be replaced by an injective one with the same set of critical simplices and matched pairs; see the beginning of the proof of Theorem 3.3 in \cite{forman_98}.} discrete Morse function $f \colon K \to \mathbb{R}$ induces a Bestvina-Brady Morse function $g \colon V(K') \to \mathbb{R}$ such that critical simplices of $f$ correspond to vertices in $K'$ whose descending links are barycentric subdivisions of the boundaries of those simplices --- and hence homeomorphic to spheres. Each matched pair $(\sigma, \tau)$, with $\sigma \prec \tau$ and $f(\sigma) > f(\tau)$, satisfies that $\sigma$ is descending dominated by $\tau$, and $\lk^\downarrow(\tau, K')$ is a topological cone on the unique vertex $v \in \tau \setminus \sigma$ (though not necessarily a simplicial one). This suggests a natural extension of our theory to cases where strong critical vertices have descending links that are topological cones.
\end{remark}

\section{The construction of the internal core}\label{sec:reconstruction core}

Given a CW-complex $K$, its face poset $\X(K)$ encodes the incidence relations among its cells: $e \leq e'$ in $\X(K)$ if $\overline{e} \subseteq \overline{e'}$. While this poset does not determine the homotopy type of $K$ in general, in the case of \emph{regular} CW-complexes it serves as a complete combinatorial model: the \textit{order complex}\footnote{The \textit{order complex} of a finite poset $X$ is the simplicial complex whose simplices are the non-empty totally ordered subsets of $X$.} of $\X(K)$, denoted $\K(\X(K))$, is homotopy equivalent to $K$ (see \cite{bjorner_85}).

Let $X$ be a finite poset. An \emph{acyclic matching} on $X$ is a collection $M$ of disjoint pairs $\{x, y\} \subseteq X$ such that the quotient set $X/_\sim$, where $x \sim y$ if $\{x, y\} \in M$, admits a partial order induced by  the original order on $X$ (the relation $[x]\leq [y]\in X/_\sim$ if $x\leq y$ in  $X$ is antisymmetric).

Discrete Morse functions on a simplicial complex $K$ are in bijection with acyclic matchings on $\X(K)$, where the unmatched elements correspond to the critical simplices \cite{chari_00}. Given an acyclic matching $M$ on $\X(K)$, one can construct auxiliary posets that reflect the structure of the Morse reduced CW-complex. Inspired by Nanda's work on discrete Morse theory and localization of categories \cite{nanda_19}, we introduce the following definitions.

\begin{definition}
Let $K$ be a simplicial complex, and let $M$ be an acyclic matching on $\X(K)$. The \emph{localization poset} $\Loc_M(K)$ is the poset obtained by identifying matched simplices in $\X(K)$. The \emph{critical poset} $\Crit_M(K)$ is the subposet of $\Loc_M(K)$ consisting of equivalence classes of unmatched (i.e., critical) simplices.
\end{definition}

In general, neither $\Crit_M(K)$ nor $\Loc_M(K)$ determine the homotopy type of $K$, even when $K$ is regular or simplicial (see Example~\ref{ex:localization_critical_poset}). However, when $M$ arises from internal strong collapses, Theorem~\ref{strong morse matching} implies that $\Crit_M(K)$ is a finite model of $K$.

\begin{corollary}
Let $K$ be a simplicial complex and let $g \colon V(K) \to \mathbb{R}$ be a real-valued function. Then, there exists an acyclic matching $M_g$ on $\X(K)$ such that the critical poset $\Crit_{M_g}(K)$ is the face poset of $\core_g(K)$.
In particular, $\Crit_{M_g}(K)$ is a finite model of $K$.
\end{corollary}

For details on the algorithmic construction of the critical poset, see Appendix \ref{appendix}, Algorithm~\ref{alg:crit}.
\iffalse
\begin{algorithm}[htb!]
\caption{The critical poset}\label{alg:crit}
\begin{algorithmic}
%\Procedure{MyProcedure}
\Require \\$X$, the face poset of a simplicial complex, \\$C$, set of critical simplices, \\$M$, acyclic matching.
\Ensure critical poset
\Function{CriticalPoset}{$X, \text{C}, \text{M}$}
\State $V \gets \text{Elements}(X)$
    \State $R \gets \text{CoverRelations}(X)$ 
    \ForAll{$(u, v) \in M$}
        \State $V \gets V \setminus \{u, v\}$
        \State $V \gets V \cup \{\text{NewVertex}\}$
        \State $\text{VertexMap} \gets \{v : v \mid v \in V\}$
    \EndFor

    \State $\text{NewCoverRelations} \gets \emptyset$
    \ForAll{$(w, z) \in R$}
        \If{$\text{VertexMap}[z] \neq \text{VertexMap}[w]$}
            \State $\text{NewCoverRelations} \gets \text{NewCoverRelations} \cup \{(\text{VertexMap}[w], \text{VertexMap}[z])\}$
        \EndIf
    \EndFor

    \State $\text{LocalizedPoset} \gets \text{Poset}(V, \text{NewCoverRelations})$
    \State  $\text{CriticalPoset} \gets \text{SubPoset}(\text{LocalizedPoset}, C)$
    \State \Return CriticalPoset
\EndFunction
\end{algorithmic}
\end{algorithm}
\fi

\begin{example}\label{ex:localization_critical_poset} Let $K$ be the boundary of the 3-simplex. Consider the acyclic matching depicted in Figures \ref{fig:boundary 3-simplex} and \ref{fig: loc and crit poset morse}. The posets $\Loc_M(K)$ and $\Crit_M(K)$ are contractible, and do not reflect the homotopy type of $K\simeq S^2$. 
\begin{figure}[htb!]
    \centering
    \includegraphics[width=0.25\linewidth]{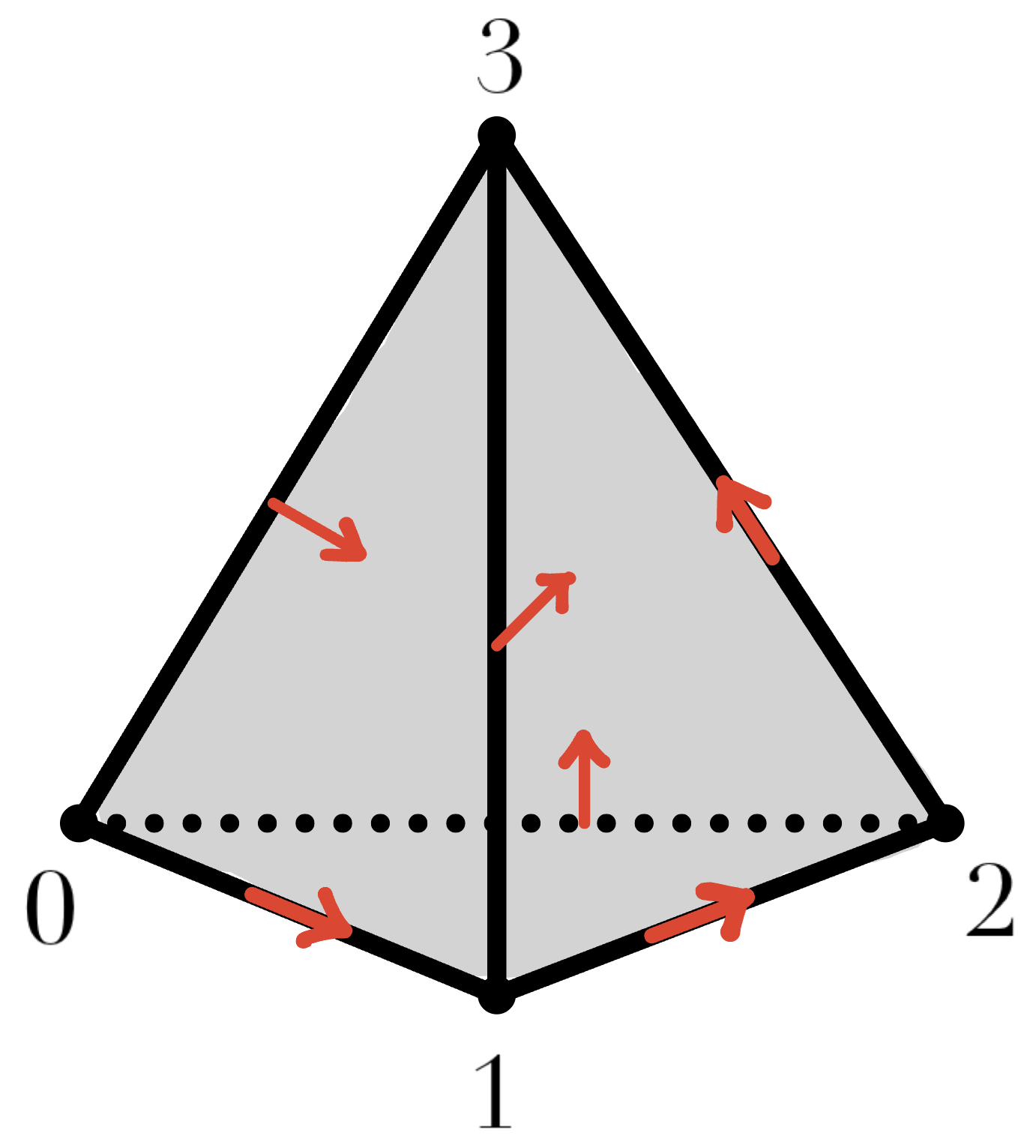}
    \caption{The boundary of a 3-simplex endowed with a function $ g \colon V(K) \to \mathbb{R} $ on vertices, together with an acyclic matching $ M $ on the simplices of $ K $. Red arrows indicate matched pairs of simplices, oriented from lower- to higher-dimensional ones. By abuse of notation, and due to the injectivity of $ g $, we identify each vertex with its image under $ g $.}
    \label{fig:boundary 3-simplex}
\end{figure}

\begin{figure}[htb!]
\[
\hspace{-10pt}\hspace{-10pt}\xymatrix@R=2.95em@C=0.3em{
   & \bullet_{(0,1,2)} \ar@{-}[dl] \ar@{-}[d] \ar@{-}[drr] & 
   \bullet_{(0,1,3)} \ar@{-}[dll] \ar@[red]@{-}[d] \ar@{-}[drr] & 
   \bullet_{(0,2,3)} \ar@[red]@{-}[dll] \ar@{-}[dl] \ar@{-}[drr] & 
   \bullet_{(1,2,3)} \ar@{-}[dl] \ar@[red]@{-}[d] \ar@{-}[dr]\\
   \bullet_{(0,1)} \ar@[red]@{-}[dr] \ar@{-}[drr] &
   \bullet_{(0,2)} \ar@{-}[d] \ar@{-}[drr] &
   \bullet_{(0,3)} \ar@{-}[dl] \ar@{-}[drr] & 
   \bullet_{(1,2)} \ar@{-}[d] \ar@[red]@{-}[dl] & 
   \bullet_{(1,3)} \ar@{-}[d] \ar@{-}[dll] & 
   \bullet_{(2,3)} \ar@{-}[dl] \ar@[red]@{-}[dll]\\
   & \bullet_{(0)} & \bullet_{(1)} & \bullet_{(2)} & \bullet_{(3)}\\
   &&\X(K)\\
}
\hspace{10pt}
\xymatrix@R=0.55em@C=0.05em{
&_{(0,1,2)}\ar@{-}[d]\\
&_{[(0,2),(0,2,3)]}\ar@{-}[d]\\
&_{[(0,3),(0,1,3)]}\ar@{-}[dl] \ar@{-}[dr]\\
_{[(0),(0,1)]}\ar@{-}[dr]&&_{[(1,3),(1,2,3)]}\ar@{-}[dl]\\
&_{[(1),(1,2)]}\ar@{-}[d]\\
&_{[(2), (2,3)]}\ar@{-}[d]\\
&_{(3)}\\
&\Loc_M(K)\\
}\hspace{10pt}
\xymatrix@R=3.25em{
_{(0,1,2)}\ar@{-}[dd]\\
\\
_{(3)}\\
\Crit_M(K)\\
}
\]
\caption{The face poset $\X(K)$, the localization poset $\Loc_M(K)$ and the critical poset $\Crit_M(K)$ associated to the acyclic matching $M$ in red (c.f. Fig \ref{fig:boundary 3-simplex}).}\label{fig: loc and crit poset morse}
\end{figure}
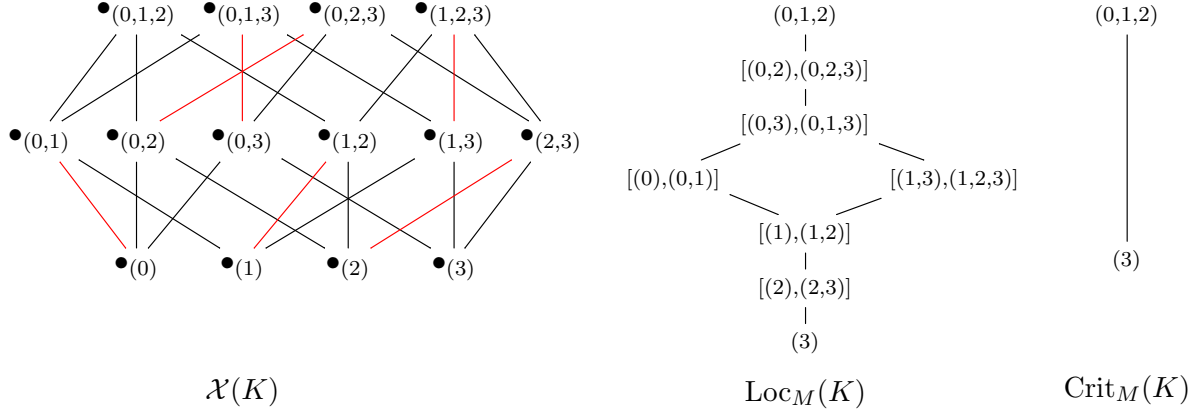

On the other hand, consider the function $g\colon V(K)\to \RR$ on the vertices depicted in Figure \ref{fig:boundary 3-simplex}. For the associated acyclic matching $M$, the  critical poset $\Crit_M(K)$ recovers the homotopy type of $K$, and it is the face poset of the strong internal core (see Figure \ref{fig:face crit poset strong morse}).
\begin{figure}[htb!]
\begin{minipage}{0.15\textwidth}
\includegraphics[width = 0.84\textwidth]{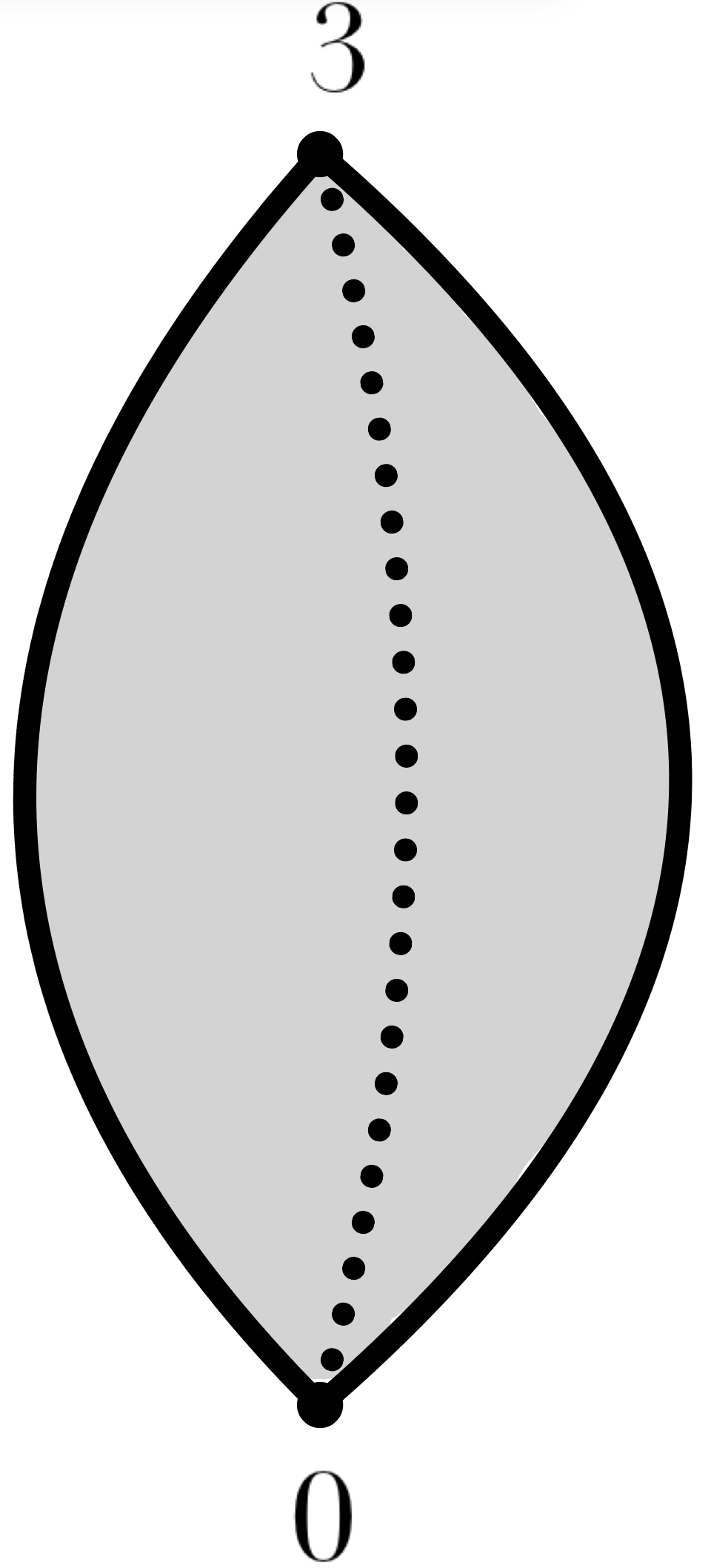}

\vspace{15pt}

$\core_g(K)$
\end{minipage}
\hspace{5pt}
\begin{minipage}{0.82\textwidth}
$\xymatrix@R=2.95em@C=0.3em{
   & \bullet_{(0,1,2)} \ar@{-}[dl] \ar@[red]@{-}[d] \ar@{-}[drr] & 
   \bullet_{(0,1,3)} \ar@{-}[dll] \ar@{-}[d] \ar@{-}[drr] & 
   \bullet_{(0,2,3)} \ar@{-}[dll] \ar@{-}[dl] \ar@{-}[drr] & 
   \bullet_{(1,2,3)} \ar@{-}[dl] \ar@{-}[d] \ar@{-}[dr]\\
   \bullet_{(0,1)} \ar@[red]@{-}[drr] \ar@{-}[dr] &
   \bullet_{(0,2)} \ar@{-}[d] \ar@{-}[drr] &
   \bullet_{(0,3)} \ar@{-}[dl] \ar@{-}[drr] & 
   \bullet_{(1,2)} \ar@[red]@{-}[d] \ar@{-}[dl] & 
   \bullet_{(1,3)} \ar@{-}[d] \ar@{-}[dll] & 
   \bullet_{(2,3)} \ar@{-}[dl] \ar@{-}[dll]\\
   & \bullet_{(0)} & \bullet_{(1)} & \bullet_{(2)} & \bullet_{(3)}\\
   &&\X(K)\\
}
\hspace{10pt}
\xymatrix@R=2.95em@C=0.3em{
   \bullet_{(0,1,3)}  \ar@{-}[d] \ar@{-}[dr] & 
   \bullet_{(0,2,3)} \ar@{-}[dl] \ar@{-}[dr] & 
   \bullet_{(1,2,3)} \ar@{-}[dl] \ar@{-}[d] \\
   \bullet_{(0,3)} \ar@{-}[d] \ar@{-}[drr] & 
   \bullet_{(1,3)} \ar@{-}[dr] \ar@{-}[dl] & 
   \bullet_{(2,3)} \ar@{-}[dll] \ar@{-}[d]\\
   \bullet_{(0)} & & \bullet_{(3)}\\
   &\Crit_{M_g}(K)\\
}
$
\end{minipage}
\caption{The strong internal core associated to the map $g\colon V(K)\to \RR$ on vertices (Fig. \ref{fig:boundary 3-simplex}), the face poset $\X(K)$ (with the acyclic matching $M_g$ induced by internal strong collapses, in red) and the critical poset $\Crit_{M_g}(K)$. Here, $\core_g(K)$ has two 0-cells, three 1-cells and  three 2-cells.}\label{fig:face crit poset strong morse}
\end{figure}

\end{example}

\begin{remark}
Consider the order-preserving maps 
\[
\Q \colon \X(K) \to \Loc_M(K) \quad \text{and} \quad \J \colon \Crit_M(K) \to \Loc_M(K),
\]
where $ \Q $ denotes the quotient map and $ \J $ the inclusion map.
 In future work, we will establish general conditions on the acyclic matching $M$ under which these maps induce homotopy equivalences at the level of simplicial complexes. In particular, we will show that when $M$ is induced by internal strong collapses, both $\Q$ and $\J$ induce homotopy equivalences.
\end{remark}

\section{Algorithmic simplification of simplicial complexes}\label{sec:random_reduction}

Let $K$ be a simplicial complex of dimension $N$. We summarize below the different core notions associated to $K$:

\begin{itemize}
  \item \textbf{Strong core:} A subcomplex $L \subseteq K$ such that $K \sco L$. The minimal strong core (i.e. with no dominated vertices) is unique up to isomorphism.

  \item \textbf{Weak core:} A subcomplex $L \subseteq K$ such that $K \co L$. A minimal weak core (i.e. with no free faces) is not necessarily unique.

  \item \textbf{Internal core:} Given a discrete Morse function $f \colon K \to \mathbb{R}$, the internal core $\core_f(K)$ is the CW-complex obtained by performing the sequence of internal collapses determined by $f$. It satisfies $K \Nplusonedef \core_f(K)$. The internal core is not necessarily regular.

  \item \textbf{Strong internal core:} Given a function $g \colon V(K) \to \mathbb{R}$ defined on the vertices of $K$, the strong internal core, $\core_g(K)$, is obtained by the sequence of internal strong collapses guided by $g$. It satisfies $K \Nplusonedef \core_g(K)$, and the resulting space is a regular CW-complex.
\end{itemize}

The diagram below shows implication relations among the core constructions.
\[
\begin{tikzcd}[row sep=small, column sep=normal]
  & \text{Internal Core} & \\
  \text{Strong Internal Core} \arrow[ur, "\text{Thm.~\ref{strong morse matching}}"] & & \text{Weak Core} \arrow[ul, "\text{Prop.~\ref{internal core vs weak core}}"'] \\
  & \text{Strong Core} \arrow[ul, "\text{Prop.~\ref{prop:strong-vs-stronginternal}}"] \arrow[ur, "\text{\protect\cite[Rmk 2.4]{barmak_minian_12}}"'] &
\end{tikzcd}
\]

For any method of reduction of simplicial complexes, computing the associated core requires the choice of an \textit{ordered} sequence of collapses. While the minimal strong core is unique (up to isomorphism) regardless of the sequence of strong collapses, a minimal weak core does depend on the chosen sequence. Similarly, the internal core of a simplicial complex depends on the discrete Morse function. Note that (strong) discrete Morse theory can be seen as a systematic procedure for performing a sequence of (strong) reductions. These reductions may consist of (strong) internal collapses, or, when no free face (resp. dominated vertex) is available, the removal of a maximal simplex (resp. vertex). 
To address the dependence on the choice of the reduction sequence, one approach is to explore all possible sequences. However, this can be computationally prohibitive due to the exponential growth in the number of options. Alternatively, random reductions --- where reductions are chosen randomly among the available options at each step --- provide an efficient strategy to deal with this dependence.

To compute the minimal strong internal core, we iteratively apply a greedy algorithm that performs randomly chosen internal strong collapses whenever a dominated vertex is available. If no dominated vertex exists, we randomly select a vertex to be removed (see Appendix \ref{appendix},  Algorithm \ref{alg:internal_core_computation} for further details). As part of this procedure, we construct the acyclic matching associated to the internal strong collapses as stated in Theorem \ref{strong morse matching}.  Notice that this algorithm can also be applied to study $(N+1)$-deformations, as the strong internal core $(N+1)$-deforms to the original complex. A similar procedure can be applied for the randomized computation of a minimal weak core of a simplicial complex.

In general, there is no combinatorial description of the internal core determined by a general discrete Morse function. Benedetti and Lutz \cite{benedetti_14} performed a computational analysis of the topology of simplicial complexes using randomized discrete Morse theory, studying the number of critical cells arising from random discrete Morse functions. %However, this analysis is not sufficient to reconstruct the homotopy type of the complex.

We analyse the examples from the \textit{Library of Triangulations} \cite{Triangulation_Library} and compare the results of minimals weak core, strong core, and strong internal core, cases where the reduced complex can be completely reconstructed.
We perform 10 iterations of the computation of random simple and strong internal cores for the simplicial complexes listed in Table \ref{tab:simplicial_complex_reduction_10it}, and 100 iterations for those in Table \ref{tab:simplicial_complex_reduction_100it}. Additionally, we implement a combined reduction strategy, where we first compute a random minimal weak core of the original complex and subsequently compute a random strong internal core of the resulting complex.
Tables \ref{tab:simplicial_complex_reduction_10it} and \ref{tab:simplicial_complex_reduction_100it}   present the mean sizes of the different cores and the corresponding computation times\footnote{The experiments were conducted on a MacBook Pro with an Apple M3 Pro chip, 12-core CPU, 18-core GPU, and 18 GB of unified memory.}. The implementation was developed in {\fontfamily{lmss}\selectfont SageMath} \cite{sagemath}, and the code is available at \url{https://github.com/ximenafernandez/Strong-Morse-Theory}.

We observe that in cases where there are neither dominated vertices nor free faces, internal strong collapses allow  significant reductions in the size of the complex, with reductions  of 30\% to 40\% in instances as \texttt{BH\_3}, \texttt{BH\_4}, \texttt{Bernadette\_Sphere}, and \texttt{non\_PL} and 50\% to 60\% in the large examples \texttt{trefoil\_bsd} and \texttt{Hom\_C5\_K4}. Notably, in many of these cases, the vector of critical simplices in a (classical) discrete Morse function does not fully determine the homotopy type. For example, the 18-vertex non-PL triangulation of the 5-dimensional sphere $S^5$ (\texttt{non\_PL}) admits $(1, 0, 0, 2, 2, 1)$ as its smallest discrete Morse vector (see \cite{benedetti_14}).
In several large cases with available simple collapses, the strong Morse theory approach yields a more efficient model, significantly reducing the complex size by around 80\% and computation times by over 90\%  in examples such as \texttt{knot} and \texttt{bing}.

On the other hand, a few cases exhibit simplicial complexes with small weak cores but large strong internal cores (on average). Examples include \texttt{B\_3\_9\_18}, \texttt{trefoil\_arc}, \texttt{rudin}, and \texttt{triple\_trefoil\_arc}. In these scenarios, the combined strategy emerges as the optimal choice: preprocessing the simplicial complex by computing its minimal weak core before determining the strong internal core.

\begin{table}[htb!]
\centering
\resizebox{\textwidth}{!}
{
\begin{tabular}{|c|r|r|r|r|r|}
\hline
\multirow{2}{*}{\textbf{Name example}} & \multicolumn{5}{c|}{\textbf{Number of simplices/Time per iteration (sec)}} \\ \cline{2-6} 
 & \textbf{Original} & \textbf{Minimal strong core} & \textbf{Minimal weak core} & \textbf{Strong internal core} & \makecell{\textbf{ Strong internal core} \\ \textbf{of minimal weak core}} \\ \hline
\textbf{\textcolor{blue}{bing}}& 8131 & 2259 & 2377.6 & \textbf{1665.80} & 1581.20 \\ 
 & -- & 185.7559  & 1532.325& \textcolor{blue}{112.5924} & 1586.5839 \\
\hline
non\_4\_2\_colorable& 5982 & 5982 & 5982.00 & 5838.20 & 5857.40 \\ 
 & -- & 2.2379 & 1.5879 & 78.1108  & 79.2957 \\
\hline
\textbf{Hom\_C5\_K4}& 6240 & 6240 & 6240.00 & \textbf{4632.40} & 4669.60 \\ 
 & -- & 3.0053  & 1.6739 & 139.9162 & 141.8681 \\
\hline
\textbf{trefoil\_bsd}& 5876 & 5876 & 5876.00 & \textbf{2072.20} & 2134.80 \\ 
 & -- & 3.7306  & 1.5139 & 102.2188& 110.9461 \\
\hline
\textbf{\textcolor{blue}{knot}}& 6203 & 1639 & 1946.20 & \textbf{1184.00} & 1271.60 \\ 
& -- & 86.4477 & 2088.4905 & \textcolor{blue}{56.3632}  & 2130.4097 \\
\hline
\end{tabular}
}
\caption{Comparison of reduction methods for simplicial complexes, averaged over 10 iterations. \textbf{Bold} values indicate cases where strong internal collapses yield the most significant size reduction. \textcolor{blue}{Blue} values indicate cases where strong internal collapses also lead to a significant improvement in execution time compared to other methods.}
\label{tab:simplicial_complex_reduction_10it}
\end{table}

\begin{table}[H]
\centering
\resizebox{\textwidth}{!}
{
\begin{tabular}{|c|r|r|r|r|r|}
\hline
\multirow{2}{*}{\textbf{Name example}} & \multicolumn{5}{c|}{\textbf{Number of simplices/Time per iteration (sec)}} \\ \cline{2-6} 
 & \textbf{Original} & \textbf{Minimal Strong core} & \textbf{Minimal weak core} & \textbf{Strong internal core} & \makecell{\textbf{ Strong internal core} \\ \textbf{of minimal weak core}} \\ \hline
\textbf{Abalone} &101 & 101 & 101.00 & \textbf{73.30} & 70.58 \\
 & -- & 0.0273 & 0.0009 & 0.0484 & 0.0486 \\
\hline
BH &131 & 131 & 131.00 & 100.68 & 99.60 \\
 & -- & 0.0185 & 0.0015 & 0.0839 & 0.0824 \\
\hline
\textbf{BH\_3} &301 & 301 & 301.00 & \textbf{206.32} & 207.78 \\
 & -- & 0.0401 & 0.0068 & 0.4438 & 0.4473 \\
\hline
\textbf{BH\_4} &401 & 401 & 401.00 & \textbf{277.06} & 274.60 \\
 & -- & 0.0564 & 0.0124 & 0.8285 & 0.8371 \\
\hline
\textbf{BH\_5} &501 & 501 & 501.00 & \textbf{349.08} & 343.86 \\
 & -- & 0.0806 & 0.0197 & 1.4011 & 1.4009 \\
 \hline
d2\_n8\_3torsion &49 & 49 & 49.00 & 38.48 & 38.38\\
 & -- & 0.0060 & 0.0003 & 0.0135 & 0.0138 \\
 \hline
d2\_n8\_4torsion &53 & 53 & 53.00 & 43.66 & 43.98 \\
 & -- & 0.0060 & 0.0003 & 0.0135 & 0.0138 \\
 \hline
d2\_n9\_5torsion &65 & 65 & 65.00 & 57.12 & 55.68 \\
 & -- & 0.0073 & 0.0005 & 0.0193 & 0.0204 \\
\hline
dunce\_hat& 49 & 49 & 49.00 & 37.60 & 38.84 \\ 
 & -- & 0.0056 & 0.0003 &  0.0135 &  0.0129 \\
 \hline
d2n12g6& 122 & 122 & 122.00 & 118.80 & 118.66 \\ 
 & -- & 0.0140 &  0.0013 &  0.00602 &  0.0612 \\
\hline
regular\_2\_21\_23\_1& 266 & 266 & 266.00 & 259.52 & 259.52 \\ 
 & -- & 0.0333 &  0.0058 &  0.2230 &  0.2349 \\
 \hline
rand2\_n25\_p0.328& 1076 & 1076 & 1074.00 & 1069.90 & 1067.18 \\ 
 & -- & 0.1965& 0.2955& 1.3532& 1.6397 \\
 \hline
dunce\_hat\_in\_3\_ball& 75 & 1 & 1.00 & 1.00 & 1.00 \\ 
 & -- & 0.0149 & 0.0151 & 0.0119 & 0.0153 \\
 \hline
\textbf{Barnette\_sphere}& 92 & 92 & 92.00 & \textbf{40.66} & 39.58 \\ 
 & -- & 0.0130 & 0.0005 & 0.0182 & 0.0195 \\
\hline
\textcolor{red}{B\_3\_9\_18}& 103 & 59 & \textcolor{red}{1.00} & 31.38 & 1.00 \\ 
& -- & 0.0220 & 0.0286 & 0.0212 & 0.0287 \\
\hline
\textcolor{red}{trefoil\_arc}& 193 & 193 & \textcolor{red}{1.76} & 148.68 & 1.62 \\ 
 & -- & 0.0281 & 0.1257 & 0.0865 & 0.1262 \\
 \hline
trefoil& 250 & 250 & 250.00 & 206.98 & 211.12 \\ 
 & -- & 0.0369 & 0.0031 & 0.1238 & 0.1298 \\
 \hline
\textcolor{red}{rudin}& 215 & 215 & \textcolor{red}{1.00} & 137.74 & 1.00 \\ 
 & -- & 0.0329 & 0.1668 & 0.1057 & 0.1670 \\
 \hline
poincare& 392 & 392 & 392.00 & 361.80 & 364.24 \\ 
 & -- & 0.0631 & 0.0074 & 0.3636 & 0.2744 \\
 \hline
double\_trefoil& 400 & 400 & 400.00 & 380.20 & 383.08 \\ 
 & -- & 0.0643 & 0.0078 & 0.2714 & 0.2797 \\
 \hline
\textcolor{red}{triple\_trefoil\_arc}& 449 & 449 &  \textcolor{red}{172.36}&  438.62& 157.60 \\ 
 & -- & 0.0746 & 1.3937 & 0.3337 & 1.5215 \\
 \hline
triple\_trefoil& 536 & 536 & 536.00 & 526.50 & 525.70 \\ 
 & -- & 0.0922 & 0.0141 & 0.4367 & 0.4467 \\
 \hline
hyperbolic\_dodecahedral\_space& 718 & 718 & 718.00 & 708.82 & 709.36 \\ 
& -- & 0.1313 & 0.0241 & 23.3319 & 74.4473 \\
\hline
S\_3\_50\_1033& 4232 & 4232 & 4232.00 & 4198.36 & 4295.20 \\ 
& -- & 1.5958 & 0.8286 & 16.9282 & 17.73918 \\
\hline
600\_cell& 2640 & 2640 & 2640.00 & 2343.46 & 2349.68 \\ 
 & -- & 0.6912 & 0.3057 & 16.5970 & 16.9279 \\
 \hline
CP2& 255 & 255 & 255.00 & 219.78 & 219.68 \\ 
 & -- & 0.0464 & 0.0020 & 0.0798 & 0.0823 \\
 \hline
RP4& 991 & 991 & 991.00 & 942.36 & 937.34 \\ 
 & -- & 0.2400 & 0.0293 & 0.7976 & 0.8290 \\
 \hline
K3\_16& 1704 & 1704 & 1704.00 & 1691.34 & 1690.62 \\ 
 & -- & 0.4895 & 0.09545 & 1.61122 & 1.7076 \\
 \hline
K3\_17& 1854 & 1854 & 1854.00 & 1835.32 & 1836.22 \\ 
 & -- & 0.6331  & 0.1108 & 1.9267 & 2.0365 \\
 \hline
RP4\_K3\_17& 2813 & 2813  & 2813.00 & 2757.98 & 2751.66 \\ 
 & -- & 0.9971  & 0.2429 & 5.2140 & 5.4425 \\
 \hline
RP4\_11S2xS2& 3179 & 3179  & 3179.00 & 3099.04 & 3105.62 \\ 
& -- & 1.2477  & 0.3071 & 6.8191 & 7.148 \\
\hline
SU2\_SO3& 1534 & 1534 & 1534.00 & 1474.70 & 1471.84 \\ 
 & -- & 0.5173 & 0.0453  & 1.2220 & 1.2677 \\
 \hline
\textbf{non\_PL}& 2608 & 2608 & 2608.0 & \textbf{1690.56} & 1539.16 \\ 
& -- & 1.1274  & 0.1274 & 2.8094 & 2.6714 \\
\hline
\end{tabular}
}

\caption{Comparison of algorithmic reduction methods for simplicial complexes, averaged over 100 iterations. \textbf{Bold} values highlight cases where significant reduction is achieved via strong internal collapses (compared to other methods). \textcolor{red}{Red} values indicate cases where significant reduction is only achieved through Whitehead collapses.}
\label{tab:simplicial_complex_reduction_100it}
\end{table}

\appendix
\section{Algorithms} \label{appendix}

This appendix contains the pseudocode of the main algorithms used  in the construction and computation of the critical poset associated with a strong internal core. 
An implementation in {\fontfamily{lmss}\selectfont SageMath} \cite{sagemath} can be found at \url{https://github.com/ximenafernandez/Strong-Morse-Theory}.

\begin{algorithm}[htb!]
\caption{The critical poset}
\label{alg:crit}
\begin{algorithmic}[1]
\Require Face poset $\mathcal{X}(K)$ of a simplicial complex $K$, set of critical simplices $C$, acyclic matching $M$ on $\mathcal{X}(K)$
\Ensure Critical poset $\Crit_M(K)$
\State $P \gets$ elements of $\mathcal{X}(K)$
\State $R \gets$ cover relations of $\mathcal{X}(K)$
\State Define a map $\QuotientClass : P \to \text{new elements}$
\ForAll{$(\sigma, \tau) \in M$}
    \State $P \gets P \setminus \{\sigma, \tau\}$
    \State $[\sigma \sim \tau] \gets$ new element representing $\{\sigma, \tau\}$
    \State $P \gets P \cup \{[\sigma \sim \tau]\}$
    \State $\QuotientClass[\sigma] \gets [\sigma \sim \tau]$
    \State $\QuotientClass[\tau] \gets [\sigma \sim \tau]$
\EndFor
\ForAll{$\sigma \in P$ such that $\QuotientClass[\sigma]$ is undefined}
    \State $\QuotientClass[\sigma] \gets \sigma$
\EndFor
\State $R' \gets \emptyset$
\ForAll{$(\sigma, \tau) \in R$}
    \If{$\QuotientClass[\sigma] \neq \QuotientClass[\tau]$}
        \State $R' \gets R' \cup \big(\QuotientClass[\sigma], \QuotientClass[\tau]\big)$
    \EndIf
\EndFor  
\State $\Loc_M(K) \gets$ poset with elements $P$ and relations $R'$
\State $\Crit_M(K) \gets$ subposet of $\Loc_M(K)$ induced by $C$
\State \Return $\Crit_M(K)$
\end{algorithmic}
\end{algorithm}

\iffalse
\begin{algorithm}[htb!]
\caption{A random strong internal core}
\label{alg:internal_core_computation}
\begin{algorithmic}
\Require Face poset $X$ of a simplicial complex
\Ensure A random strong internal core of $X$
\Function{StrongMorseReduction}{$X, \text{C}, \text{M}$}
    \If{$X = \emptyset$}
        \State \Return $X, \text{C}, \text{M}$
    \EndIf
    \State $\text{Minimal} \gets \text{MinimalElements}(X)$ (shuffled)
    \For{$v \in \text{Minimal}$}
        \If{$v$ is dominated}
            \State $a\gets$ dominating vertex
            \State $\text{M} \gets \text{M} \cup \{(\sigma, a\sigma)\}$ for $(\sigma, a\sigma)$ for $\sigma$ in $\text{DescendingOpenStar}(v)$ such that  $a\notin \sigma$
            \State $X \gets X \setminus \{v, \text{all simplices containing } v\}$
            \State \Return \Call{StrongMorseReduction}{$X, \text{C}, \text{M}$}
        \EndIf
    \EndFor
    \State $v \gets \text{FirstMinimalElement}(X)$
    \State $\text{C} \gets \text{C} \cup \text{DescendingOpenStar}(v)$
    \State $X \gets X \setminus \{v, \text{all simplices containing } v\}$
    \State \Return \Call{StrongMorseReduction}{$X, \text{C}, \text{M}$}
\EndFunction
\Procedure{ComputeInternalStrongCore}{$X$}
    \State $\text{C} \gets \emptyset$, $\text{M} \gets \emptyset$
    \State $Y, \text{C}, \text{M} \gets$ \Call{StrongMorseReduction}{$X, \text{C}, \text{M}$}
    \State \Return \Call{CriticalPoset}{$X, \text{C}, \text{M}$}
\EndProcedure
\end{algorithmic}
\end{algorithm}
\fi

\begin{algorithm}[htb!]
\caption{A random strong internal core}
\label{alg:internal_core_computation}
\begin{algorithmic}
\Require $\X(K)$: the face poset of a simplicial complex $K$
\Ensure The critical poset of a random strong internal core of $K$
\\
\Function{StrongMorseReduction}{$X, C, M$} \\\Comment{{\footnotesize{Auxiliary function. Input: $X$: face poset, $M$: acyclic matching, $C$: critical elements}}}
    \If{$X = \emptyset$}
        \State \Return $(X, C, M)$
    \EndIf
    \State $S \gets \text{Shuffle}(\text{MinimalElements}(X))$ \Comment{{\footnotesize Randomized linear extension of minimal elements}}
    \For{$v \in S$}
        \If{there exists $a \in \text{MinimalElements}(X) \setminus \{v\}$ such that $a$ dominates $v$}
            \ForAll{$\sigma \in \{\sigma \in X\colon v \leq \sigma\}$ such that $a \nleq \sigma$}
                \State $M \gets M \cup \{ (\sigma, \{a\}\cup\sigma) \}$
            \EndFor
            \State $X \gets X \setminus \{ \tau \in X \colon v \leq \tau \}$
            \State \Return \Call{StrongMorseReduction}{$X, C, M$}
        \EndIf
    \EndFor
    \State $v \gets \text{First}(S)$ \Comment{{\footnotesize Take first non-dominated minimal element}}
    \State $C \gets C \cup \{\sigma \in X\colon v \leq \sigma\}$
    \State $X \gets X \setminus \{\sigma \in X\colon v \leq \sigma\}$
    \State \Return \Call{StrongMorseReduction}{$X, C, M$}
\EndFunction
\\
\Procedure{ComputeInternalStrongCore}{$\X(K)$}
    \State $(-, C, M) \gets$ \Call{StrongMorseReduction}{$\X(K), \emptyset, \emptyset$}
    \State \Return \Call{CriticalPoset}{$\X(K), C, M$}
\EndProcedure
\end{algorithmic}
\end{algorithm}

\begin{remark}
    The worst-case complexity of a naive implementation of Algorithm \ref{alg:internal_core_computation} is $O(V^3 \times F)$, where $V$ is the number of vertices of the simplicial complex $K$, and $F$ is the number of simplices of $K$. 
\end{remark}

\section*{Acknowledgments}

The author thanks Robert Green for his enthusiasm for posets and for pointing out the earlier work of Nicholas Scoville on strong Morse theory during his visit to Oxford. This work was inspired by fruitful conversations with him at the Mathematical Institute. The author is also grateful to Vidit Nanda for engaging with early versions of this project and for suggesting the connection with the localization of categories, and to Gabriel Minian for useful comments on the connection with Bestvina–Brady Morse theory, and to Eugenio Borghini for valuable discussions. 
This work was partially supported by Leverhulme Trust Research Project Grant RPG-2023-144 and the UK Centre for Topological Data Analysis EPSRC grant EP/R018472/1.

%\section{Strong Morse theory for homotopy persistence}

\bibliographystyle{plain}
\bibliography{biblio}
\end{document}